\documentclass{article}
\usepackage{amssymb}
\usepackage{amsmath}
\usepackage[margin=1.10in]{geometry}

\newtheorem{theorem}{Theorem}

\newtheorem{definition}[theorem]{Definition}

\newtheorem{lemma}[theorem]{Lemma}

\newtheorem{remark}[theorem]{Remark}

\newenvironment{proof}[1][Proof]{\noindent\textbf{#1.} }{\ \rule{0.5em}{0.5em}}
\input{tcilatex}

\begin{document}

\title{The theory of rough paths via one-forms and the extension of an
argument of Schwartz to rough differential equations}
\author{Terry J. Lyons, Danyu Yang \thanks{%
The authors would like to acknowledge the support of the Oxford-Man
Institute and the support provided by ERC advanced grant ESig (agreement no.
291244).}}
\maketitle

\begin{abstract}
We give an overview of the recent approach to the integration of rough paths
that reduces the problem to classical Young integration \cite%
{young1936inequality}. As an application, we extend an argument of Schwartz 
\cite{schwartz1989convergence} to rough differential equations, and prove
the existence, uniqueness and continuity of the solution, which is
applicable when the driving path takes values in nilpotent Lie group or
Butcher group.
\end{abstract}

\section{Overview}

For each $p\in \left[ 1,\infty \right) $ Banach introduced a metric for
measuring degrees of roughness in paths with values in Banach spaces known
as $p$-variation. The paths of finite $1$-variation are dense in the space
of paths of finite $p$-variation for each $p\geq 1$. Where when $p=1$ the
paths are weakly differentiable almost surely and they engage with the
classical Newtonian calculus for example making sense of line integrals:%
\begin{equation*}
\tint\nolimits_{t\in \left[ 0,T\right] }\tau _{t}\otimes d\sigma _{t}.
\end{equation*}%
Young \cite{young1936inequality} extended the integration so that if $\tau $
has finite $q$-variation and $\sigma $ is continuous\footnote{%
or at least has its jumps in different times to $\tau $} and has finite $p$%
-variation where $p^{-1}+q^{-1}>1$ then 
\begin{equation*}
\tint \tau \otimes d\sigma 
\end{equation*}%
is well defined. In particular, if $\sigma $ is of finite $p$-variation for $%
p<2$ then the integral 
\begin{equation*}
\tint \sigma \otimes d\sigma 
\end{equation*}%
is meaningfully defined. Young's original definition was directed towards
definite integrals. Lyons \cite{lyons1994differential} considered the case
of indefinite integrals and the related context of controlled systems of
differential equations:%
\begin{equation}
dy_{t}=f\left( y_{t}\right) d\sigma _{t},\ y_{0}=a,
\label{first differential equation}
\end{equation}%
established the existence and uniqueness of the solution, and also the
continuity of the solution in the driving signal. Lyons' integral requires
the finite $p$-variation of $\sigma $, the finite $\limfunc{Lip}\left(
\gamma \right) $ norm of $f$, and $p^{-1}+\gamma p^{-1}>1$. The methods rely
strongly on Young's approach, but a careful examination reveals that the
arguments also rely critically on the notion of the Lipschitz function and
on the division lemma for them (Proposition 1.26 \cite{lyons2007differential}%
).

\begin{lemma}[Division Property]
\label{Lemma division property}For Banach spaces $\mathcal{U}$ and $\mathcal{%
W}$, suppose $f:\mathcal{U}\rightarrow \mathcal{W}$ is $\limfunc{Lip}\left(
\gamma \right) $ for some $\gamma >1$. Then there exists $h:\mathcal{U}%
\times \mathcal{U}\rightarrow L\left( \mathcal{U},\mathcal{W}\right) $ which
is $\limfunc{Lip}\left( \gamma -1\right) $ such that%
\begin{equation*}
f\left( x\right) -f\left( y\right) =h\left( x,y\right) \left( x-y\right) ,\
\forall x,y\in \mathcal{U}\text{,}
\end{equation*}%
and for some constant $C$ depending only on $\gamma $ and $\mathcal{U}$,%
\begin{equation*}
\left\Vert h\right\Vert _{\limfunc{Lip}\left( \gamma -1\right) }\leq
C\left\Vert f\right\Vert _{\limfunc{Lip}\left( \gamma \right) }\text{.}
\end{equation*}
\end{lemma}

The bound $p<2$ becomes an essential part of the thinking if one relies on
Young's integral. Both $p$-variation paths and $\limfunc{Lip}\left( \gamma
\right) $ functions form local algebras, and $y$ in $\left( \ref{first
differential equation}\right) $ also has finite $p$-variation. From this it
is clear that the space of integrals of $\sigma $, including all spaces of
solutions to differential equations driven by $\sigma $, is closed under
addition, and the pointwise multiplication is explicitly given by 
\begin{gather*}
\text{for }y_{t}=a+\int_{s\in \left[ 0,t\right] }f\left( y_{s}\right)
d\sigma _{s}\text{ and }\hat{y}_{t}=\hat{a}+\int_{s\in \left[ 0,t\right] }%
\hat{f}\left( \hat{y}_{s}\right) d\sigma _{s}\text{,} \\
y_{t}\hat{y}_{t}=\int_{s\in \left[ 0,t\right] }\left( f\left( y_{s}\right) 
\hat{y}_{s}+y_{s}\hat{f}\left( \hat{y}_{s}\right) \right) d\sigma _{s}+a\hat{%
a}\text{ .}
\end{gather*}%
This remark is implicit in establishing the existence, uniqueness and
continuity theorems since it underpins the operations used in Picard
iteration and other approximation strategies. In fact it is easy to show
that composition of an integral of $\sigma $ with a smooth function is also
an integral of $\sigma $ (the chain rule).

In further work \cite{lyons1998differential}, Lyons extended the integral of
Young to the case $p\geq 2$, showed how the notion of bounded variation
paths naturally admits a generalization to $p$-rough paths for any $p\in %
\left[ 1,\infty \right) $, and established an integral, existence,
uniqueness and continuity theorem for differential equations controlled by
weak geometric $p$-rough paths when $f$ is $\limfunc{Lip}\left( \gamma
\right) $ and $\gamma >p$. Young's tricks, the division lemma and the
algebraic manipulations of Picard iteration were all important ingredients.
The main surprise over the case $p<2$ came from the essential nonlinear
aspects of the metric imposed on bounded variation functions that allowed
the $p$-roughness. The space is quite different to that envisaged by Banach.

In this short note we summarize a new approach to the case $p\geq 2$, which
could be viewed as a proper extension of Lyons' original approach, and is
somewhere between the original arguments which emphasized the rough paths
and the perspective of Gubinelli which emphasized more the space of possible
integrands for a given path that (in his context) are referred to as
controlled rough paths. We explain how a clear perspective about a Lipschitz
function $f$ which allows one to (quite simply) reduce the problem of
defining a rough line integral 
\begin{equation*}
\int_{s\in \left[ 0,t\right] }f\left( \sigma _{s}\right) d\sigma _{s}
\end{equation*}%
to the integral of a slowly varying one-form $t\rightarrow \hat{f}\left(
\sigma _{t}\right) $ against a rapidly varying path $\sigma _{t}$ in a way
that satisfies Young's conditions.

The key understanding comes from repositioning the integral so that $\sigma $
is a path in a nilpotent group\ and $h_{t}=\hat{f}\left( \sigma _{t}\right) $
is a closed one-form on that group that varies more slowly with time than $%
\sigma $. When looked at in the correct way, Young's strategy applies and 
\begin{equation*}
\int_{s\in \left[ 0,t\right] }h_{s}d\sigma _{s}
\end{equation*}%
is well defined. Apart from the clarity this understanding gives, it
captures the linearity of the integral against a path in a convenient way,
and actually leads to the introduction of the integral of any $q$-variation
path with values in the closed one-forms against $\sigma $. It is not
surprising that the class of these integrals is again closed under addition,
pointwise multiplication and composition with smooth functions. What is more
surprising is that it is (by construction) rich enough to include the
original integral 
\begin{equation*}
\int_{s\in \left[ 0,t\right] }f\left( \sigma _{s}\right) d\sigma _{s}.
\end{equation*}%
As a result, differential equations against rough paths, etc. are easily
deduced. It is surprising because $s\mapsto f\left( \sigma _{s}\right) $ is
certainly not in general\ of finite $q$-variation for any $q$ satisfying 
\begin{equation*}
\frac{1}{p}+\frac{1}{q}>1\text{,}
\end{equation*}%
if $p\geq 2$.

The key point is actually rooted in geometry that does not have anything
(directly) to do with rough paths but it positions one accurately to do the
analysis of rough paths. We need a number of separate ingredients to explain
clearly the framework.

\subparagraph{Polynomial functions}

A \emph{polynomial function} of degree $n$ is a globally defined function
whose $\left( n+1\right) $th derivative exists and is identically zero. We
intentionally avoid the definition as a power series around a point, and we
could choose different reference points and have different representations
of the \emph{same} polynomial. More specifically, for Banach spaces $%
\mathcal{V}$ and $\mathcal{U}$, we say $p:\mathcal{V}\rightarrow \mathcal{U}$
is a polynomial function of degree (at most) $n$ if $D^{n+1}p\equiv 0$. For
any $y\in \mathcal{V}$, we can represent $p$ as a power series around $y$: 
\begin{equation*}
p\left( x\right) =\sum_{k=0}^{n}\left( D^{k}p\right) \left( y\right) \frac{%
\left( x-y\right) ^{\otimes k}}{k!}\text{, }\forall x\in \mathcal{V}\text{, }%
\forall y\in \mathcal{V}\text{,}
\end{equation*}%
but the value of $p$ does not vary with $y$. We would like to emphasize that 
$p$ is a function defined on the affine space $\mathcal{V}$, it has no
natural graded algebraic structure, there is no particular choice of base
point associated with it, and there does not exist a translation invariant
norm on the space of polynomial functions.

Just as in linear algebra, where one keeps the concept of linear map
separated from the matrix one gets after fixing a particular choice of
basis, it is conceptually essential that we distinguish the polynomial
function as an object from any representation of it via its Taylor series
around a chosen point.

For Banach space $\mathcal{U}$ and integer $n\geq 0$, let $P^{\left(
n\right) }\left( \mathcal{U}\right) $ denote the space of polynomial
functions of degree $n$ taking values in $\mathcal{U}$.

\subparagraph{Lipschitz functions}

By using the polynomial functions (rather than power series), we can shift
the classical viewpoint of the Lipschitz function as a function taking
values in power series to a function taking values in polynomial functions.
This modification gives rise naturally to a way to compare the
representations of polynomial functions, and reduces a Lipschitz function to
a "slowly-varying" polynomial function. The first author would like to thank
Youness Boutaib for sharing his understanding of Lipschitz functions with
him.

\begin{definition}[Stein]
Let $\mathcal{V}$ and $\mathcal{U}$ be two Banach spaces. For $\gamma >0$,
denote $n:=\lfloor \gamma \rfloor $ (the largest integer which is strictly
less than $\gamma $). For a closed set $\mathcal{K}$ in $\mathcal{V}$, we
say $f\ $is a\emph{\ Lipschitz function of degree }$\gamma $ on $\mathcal{K}$%
, if 
\begin{equation*}
f:\mathcal{K}\rightarrow P^{\left( n\right) }\left( \mathcal{U}\right) \text{%
,}
\end{equation*}%
and for some constant $M>0$,%
\begin{equation*}
\sup_{x\in \mathcal{K}}\left\Vert f\left( x\right) _{x}\right\Vert _{\infty
}+\sup_{x,y\in \mathcal{K}}\max_{j=0,1,\dots ,n}\left\Vert \frac{\left(
D^{j}\left( f\left( x\right) -f\left( y\right) \right) \right) _{x}}{%
\left\Vert x-y\right\Vert ^{\gamma -j}}\right\Vert _{\infty }\leq M.
\end{equation*}
\end{definition}

Some explanatory points are in order:

\begin{enumerate}
\item For $x\in \mathcal{K}$, $f\left( x\right) $ is a polynomial function
of degree $n$, and we denote by $f\left( x\right) _{x}$ the degree-$n$
Taylor series of $f\left( x\right) $ around $x$. Similarly, for $j=0,1,\dots
,n$, $\left( D^{j}\left( f\left( x\right) -f\left( y\right) \right) \right) $
is a polynomial function of degree $n-j$ and $\left( D^{j}\left( f\left(
x\right) -f\left( y\right) \right) \right) _{x}$ denotes its degree-$\left(
n-j\right) $ Taylor series around $x$.

\item For each $x\in \mathcal{K}$, $f\left( y\right) \mapsto \left\Vert
f\left( y\right) _{x}\right\Vert _{\infty }$ is a norm on $P^{\left(
n\right) }\left( \mathcal{U}\right) $. These norms are equivalent, and if $%
\mathcal{K}$ is compact then they are uniformly equivalent.

\item The $\limfunc{Lip}\left( \gamma \right) $ norm $\left\Vert
f\right\Vert _{\limfunc{Lip}\left( \gamma \right) }$ is defined to be the
smallest $M$ satisfying the inequality.

\item Suppose $\mathcal{N}$ is a neighborhood of $x$ and $\mathcal{N}%
\subseteq \mathcal{K}$. Then $F:\mathcal{N}\rightarrow \mathcal{U}$ defined
by $y\mapsto \left( f\left( y\right) \right) \left( y\right) $ for $y\in 
\mathcal{N}$ is a $C^{\gamma }$ function ($n$ times differentiable with the $%
n$th derivative $\left( \gamma -n\right) $-H\"{o}lder) and $f\left( x\right) 
$ is the polynomial function that matches $F$ to degree $n$ at $x$ :%
\begin{equation*}
\left( D^{j}\left( f\left( x\right) -F\right) \right) \left( x\right) =0,\
j=0,1,\dots ,n\text{.}
\end{equation*}%
While in comparison with the notion of $C^{\gamma }$ functions, Lipschitz
functions make perfect sense even when $\mathcal{K}$ is of finite
cardinality.

\item The space of Lipschitz functions forms an algebra.

\item Whitney's extension theorem was extended by Stein \cite%
{stein1970singular} to these generalized Lipschitz functions. He proved that
there is a constant $C_{d}$ and a linear extension operator so that any $%
\limfunc{Lip}\left( \gamma \right) $ function $f$ on a closed set $\mathcal{K%
}$ in $\mathbb{R}^{d}$ can be extended to a $\limfunc{Lip}\left( \gamma
\right) $ function $g$ on $\mathbb{R}^{d}$ where $\left\Vert g\right\Vert _{%
\mathrm{\limfunc{Lip}}\left( \gamma \right) }\leq C_{d}\left\Vert
f\right\Vert _{\mathrm{\limfunc{Lip}}\left( \gamma \right) }$.
\end{enumerate}

The \emph{crucial }and somewhat counter-intuitive remark associated with
Lipschitz functions is the following.

\begin{remark}
Suppose $p$ is a polynomial function of degree $m$ and $\gamma >0$ is a real
number. When $\gamma >m$, $p$ is associated with a \emph{constant} $\limfunc{%
Lip}\left( \gamma \right) $ function $f:\mathcal{K}\rightarrow P^{\left(
m\right) }\left( \mathcal{U}\right) $ defined by 
\begin{equation*}
f\left( x\right) :=p\text{, }\forall x\in \mathcal{K}\text{.}
\end{equation*}%
When $\gamma \leq m$, $p$ gives rise to a non-constant $\limfunc{Lip}\left(
\gamma \right) $ function%
\begin{equation*}
f\left( x\right) _{x}\left( z\right) =\sum_{l=0}^{\lfloor \gamma \rfloor
}\left( D^{l}p\right) \left( x\right) \frac{\left( z-x\right) ^{\otimes l}}{%
l!}\text{, }\forall z\in \mathcal{V}\text{, }\forall x\in \mathcal{K}\text{,}
\end{equation*}%
since $\lfloor \gamma \rfloor <m$.
\end{remark}

\begin{remark}
This transformation of polynomials into constant functions in a different
function space, and more generally, smooth functions into slowly changing
functions, can be seen at the heart of the success of the rough path
integral. Rough path integration traditionally integrates a $\limfunc{Lip}%
\left( p+\varepsilon -1\right) $ one-form against a (weak geometric) $p$%
-rough path.
\end{remark}

\subparagraph{Lifting of polynomial one-forms to closed one-forms}

For integer $n\geq 1$, the step-$n$ nilpotent Lie group $G^{n}$ has a
natural graded algebraic structure, and accommodates weak geometric $p$%
-rough paths for $p<n+1$. $G^{1}\ $is an abelian group which is isomorphic
to a Banach space, and fits naturally into the chain $G^{0}=\left\{
e\right\} \overset{\pi }{\twoheadleftarrow }G^{1}\overset{\pi }{%
\twoheadleftarrow }\ldots \overset{\pi }{\twoheadleftarrow }G^{n}\overset{%
\pi }{\twoheadleftarrow }\ldots $ . If $\sigma $ is a path of finite length
taking values in $G^{1}$, then there is a natural lift $\sigma \mapsto \hat{%
\sigma}$ (the signature of $\sigma $), which takes a path in $G^{1}$ into a
horizontal path in $G^{n}$.

We have defined polynomial functions and Lipschitz functions. A \emph{%
polynomial one-form} or a \emph{Lipschitz one-form} is a polynomial function
or Lipschitz function taking values in one-forms.

Suppose $p$ is a polynomial one-form on $G^{1}$, and we would like to lift $%
p $ to a one-form $p^{\ast }$ on $G^{n}$ so that%
\begin{equation*}
\int p\left( \sigma \right) d\sigma =\int p^{\ast }\left( \hat{\sigma}%
\right) d\hat{\sigma}\text{.}
\end{equation*}%
A simple choice is to let $p^{\ast }$ be the pullback of $p$ through the
projection $\pi $. Then the equality holds because $\sigma =\pi \hat{\sigma}$
and has nothing to do with the fact that $\hat{\sigma}$ is the "horizontal
lift" of $\sigma $. Actually, being a "horizontal lift" adds an extra
ingredient which we will exploit in a crucial way. If $\omega $ is any
one-form on $G^{n}$ which has the horizontal directions in its kernel, then 
\begin{equation*}
\int p^{\ast }\left( \hat{\sigma}\right) d\hat{\sigma}=\int \left( p^{\ast
}+\omega \right) \left( \hat{\sigma}\right) d\hat{\sigma}.
\end{equation*}%
The key point is that we can select $\omega $ such that $p^{\ast }+\omega $
is a closed one-form, and the selection only depends on $p$ and not on $\hat{%
\sigma}$.

\begin{theorem}
\label{Theorem polynomial one-form as closed one-form on group}For $n\geq 1$
and a polynomial one-form $p$ of degree $n-1$, there exists a unique
one-form $\omega $ on $G^{n}$, which is orthogonal to the horizontal
directions and $p^{\ast }+\omega $ is a closed one-form on $G^{n}$.\ 
\end{theorem}

The proof of this theorem is actually not hard: we can give one possible
choice of $\omega $, and since $p^{\ast }+\omega $ does not depend on $\hat{%
\sigma}$, any two choices must coincide.

While we should specify what we mean by a closed one-form on a group.
Roughly speaking, closed one-forms are characterized by zero integral along
closed curves, and a one-form on a connected domain is closed if it can be
integrated against any continuous path on the domain, and the value of the
integral only depends on the end points of the path. A one-form is closed is
equivalent to the exact equality between the one-step and two-steps
estimates. Integrals often correspond to closed one-forms because of the
property $\int_{\left[ s.t\right] }=\int_{\left[ s,u\right] }+\int_{\left[
u,t\right] }$, and this property is actually behind the fact that the lifted
polynomial one-form is closed. In term of mathematical expression, we say $%
\beta $ on group $\mathcal{G}$ taking values in another algebra is closed
(or cocyclic), if 
\begin{equation*}
\beta \left( a,b\right) \beta \left( ab,c\right) =\beta \left( a,bc\right) 
\text{, }\forall a,b,c\in \mathcal{G}\text{.}
\end{equation*}

By lifting a path to a horizontal path and a polynomial one-form to a closed
one-form on the nilpotent Lie group, we replace a general integral by the
integral of a closed one-form. The integral of a closed one-form has the
nice property that it does not depend on the fine structure of the path but
only on its end points. In particular, the integral makes sense for any
continuous path and has no (further) regularity assumption.

\subparagraph{Integrating slowly-varying closed one-forms}

Since the integral of a closed one-form against any continuous path is
well-defined, we could weaken the requirement on the one-form and strengthen
the regularity assumption on the path in such a way that the integral still
makes sense. For example, in the case of classical integral, we can
integrate a constant one-form against any continuous path because constant
one-forms are closed. Then if we weaken the requirement on the one-form and
strengthen the requirement on the path in such a way that their regularities
"compensate" each other, then the integral still makes sense as Young
integral \cite{young1936inequality}. In the case of Young integral, we
actually vary the constant one-form with time and get a path taking values
in constant one-forms, which is more clearly seen in the proof of the
existence of the integral where we keep comparing the constant one-forms
from different times based on their effect on the future increment of the
driving path.

Constant one-form on Banach space is just a special example of closed
one-forms. More generally, suppose we have a family of closed one-forms on a
differential manifold or on a topological group. For a given path taking
values in the manifold or group, if the closed one-form varies with time in
such a way that the one-form and the path have compensated Young
regularities, then the integral should still makes sense.

As we mentioned above, a Lipschitz one-form could be viewed as a
slowly-varying polynomial one-form, and that there exists a canonical lift
of a polynomial one-form to a closed one-form on the nilpotent Lie group.
Hence we can lift a Lipschitz one-form to a slowly-varying closed one-form
on the nilpotent Lie group. More specifically, suppose $\alpha $ is a
Lipschitz one-form on $G^{1}$. Then based on our argument above, $\alpha $
can be viewed as a slowly-varying polynomial one-form. Suppose $\sigma $ is
an underlying reference path. Then the evolution of $\sigma $ gives a
natural order (or say time), and $\alpha $ along $\sigma $ is a
"slowly-time-varying" polynomial one-form with each $\alpha _{\sigma _{t}}$
a polynomial one-form. If we denote by $\hat{\sigma}_{t}\in G^{n}$ the
horizontal lift of the path $\sigma _{t}\in G^{1}$ and denote by $\beta _{%
\hat{\sigma}_{t}}$ the closed one-form lift of the polynomial one-form $%
\alpha _{\sigma _{t}}$, then we can rewrite the integral of a Lipschitz
one-form against $\sigma $ as the integral of a time-varying closed one-form
against $\hat{\sigma}$ : 
\begin{equation*}
\int \alpha \left( \sigma _{t}\right) d\sigma _{t}=\int \alpha _{\sigma
_{t}}\left( \sigma _{t}\right) d\sigma _{t}=\int \beta _{\hat{\sigma}%
_{t}}\left( \hat{\sigma}_{t}\right) d\hat{\sigma}_{t}.
\end{equation*}

When $\sigma $ is of finite length, this algebraic/geometrical reformulation
seems unnecessary. The point is that for general path $\hat{\sigma}$ of
finite $p$-variation taking values in $G^{\left[ p\right] }$, the integral $%
\int \beta _{\hat{\sigma}}\left( \hat{\sigma}\right) d\hat{\sigma}$ still
makes sense (the rough integral) while the classical Riemann sum integral $%
\int \alpha \left( \sigma \right) d\sigma $ does not have a proper meaning.

\begin{theorem}
Suppose $\alpha $ is a $\limfunc{Lip}\left( p+\epsilon -1\right) $ one-form
for some $\epsilon >0$.\ Then there exists $\beta $ taking values in closed
(or say cocyclic) one-forms on $G^{\left[ p\right] }$, such that for any $%
\sigma _{t}\in G^{1}$ of finite length with horizontal lift $\hat{\sigma}%
_{t}\in G^{\left[ p\right] }$, we have%
\begin{equation*}
\int \alpha \left( \sigma _{t}\right) d\sigma _{t}=\int \beta _{\hat{\sigma}%
_{t}}\left( \hat{\sigma}_{t}\right) d\hat{\sigma}_{t}\text{,}
\end{equation*}%
Moreover, the integral $\int \beta _{\hat{\sigma}_{t}}\left( \hat{\sigma}%
_{t}\right) d\hat{\sigma}_{t}$ is well-defined for any continuous path $\hat{%
\sigma}$ of finite $p$-variation taking values in $G^{\left[ p\right] }$ and
the integral is continuous with respect to $\hat{\sigma}$ in $p$-variation
metric.
\end{theorem}

\subparagraph{Conclusion}

Based on our formulation, to make sense of the rough integral, all we need
is the compensated Young regularity between two dual paths: one takes values
in the group and the other takes values in the closed (cocyclic) one-forms
on the group. By viewing the Lipschitz functions as slowly-varying
polynomial functions and by lifting the polynomial one-forms to closed
one-forms, we encapsulate the nonlinearity of the integral to the structure
of the group and to the closed one-forms on the group so that\ the idea
behind the generalized integral is clearer and bears a similar form to the
linear Young integral.

\section{Definitions and Properties}

Suppose $\mathcal{U}$, $\mathcal{V}$ and $\mathcal{W}$ are Banach spaces and 
$p\geq 1$ a real number. We restate the definition of the cocyclic one-form
and the dominated path as in \cite{lyons2014integration}.

Suppose $\mathcal{A}$ and $\mathcal{B}$ are Banach algebras and $\mathcal{G}$
is a topological group in $\mathcal{A}$. We denote by $L\left( \mathcal{A},%
\mathcal{B}\right) $ the set of continuous linear mappings from $\mathcal{A}$
to $\mathcal{B}$, and we denote by $C\left( \mathcal{G},L\left( \mathcal{A},%
\mathcal{B}\right) \right) $ the set of continuous mappings from $\mathcal{G}
$ to $L\left( \mathcal{A},\mathcal{B}\right) $.

\begin{definition}[Cocyclic One-Form]
We say $\beta \in C\left( \mathcal{G},L\left( \mathcal{A},\mathcal{B}\right)
\right) $ is a cocyclic one-form, if there exists a topological group $%
\mathcal{H}$ in $\mathcal{B}$ such that $\beta \left( a,b\right) \in 
\mathcal{H}$ for all $a,b\in \mathcal{G}$ and%
\begin{equation*}
\beta \left( a,b\right) \beta \left( ab,c\right) =\beta \left( a,bc\right) 
\text{, }\forall a,b,c\in \mathcal{G}\text{.}
\end{equation*}%
We denote the set of cocyclic one-forms by $B\left( \mathcal{G},\mathcal{H}%
\right) $ (or $B\left( \mathcal{G}\right) $).
\end{definition}

Since a Banach space $\mathcal{U}$ is canonically embedded in the Banach
algebra $\left\{ \left( c,u\right) |c\in 
\mathbb{R}
,u\in \mathcal{U}\right\} $ with multiplication $\left( c,u\right) \left(
r,v\right) =\left( cr,ru+cv\right) $, we denote by $B\left( \mathcal{G},%
\mathcal{U}\right) $ the set of cocyclic one-forms taking values in $%
\mathcal{U}$ satisfying $\beta \left( a,b\right) +\beta \left( ab,c\right)
=\beta \left( a,bc\right) $ for all $a,b,c$ in $\mathcal{G}$.

For $p\geq 1$, we denote by $\left[ p\right] $ the integer part of $p$. As
in \cite{lyons2014integration}, we equip the tensor powers of $\mathcal{V}$
with admissible norms and assume $T^{\left( \left[ p\right] \right) }\left( 
\mathcal{V}\right) =%
\mathbb{R}
\oplus \mathcal{V}\oplus \cdots \oplus \mathcal{V}^{\otimes \left[ p\right]
} $ is a graded Banach algebra equipped with the norm $\left\Vert \cdot
\right\Vert :=\sum_{k=0}^{\left[ p\right] }\left\Vert \pi _{k}\left( \cdot
\right) \right\Vert $ ($\pi _{k}$ denotes the projection to $\mathcal{V}%
^{\otimes k}$), and the multiplication on $T^{\left( \left[ p\right] \right)
}\left( \mathcal{V}\right) $ is induced by a finite family of linear
projective mappings denoted by $\mathcal{P}_{\left[ p\right] }$; $\mathcal{G}%
_{\left[ p\right] }$ is a closed topological group in $T^{\left( \left[ p%
\right] \right) }\left( \mathcal{V}\right) $ whose linear span is $T^{\left( %
\left[ p\right] \right) }\left( \mathcal{V}\right) $ and whose projection to 
$%
\mathbb{R}
$ is $1$.

When $\mathcal{G}_{\left[ p\right] }$ is the nilpotent Lie group over $%
\mathcal{V}$, $\mathcal{P}_{\left[ p\right] }=\left\{ \pi _{k}\right\}
_{k=0}^{\left[ p\right] }$ with $\pi _{k}\left( ab\right) =\sum_{j=0}^{k}\pi
_{j}\left( a\right) \otimes \pi _{k-j}\left( b\right) $ for $k=0,1,\dots ,%
\left[ p\right] $ and for $a,b\in T^{\left( \left[ p\right] \right) }\left( 
\mathcal{V}\right) $. When $\mathcal{G}_{\left[ p\right] }$ is the Butcher
group over $%
\mathbb{R}
^{d}$, $\mathcal{P}_{\left[ p\right] }$ is the set of labelled forests of
degree less or equal to $\left[ p\right] $ and $\sigma \left( ab\right)
=\sum_{c}P^{c}\left( \sigma \right) \left( a\right) R^{c}\left( \sigma
\right) \left( b\right) $ for $\sigma \in \mathcal{P}_{\left[ p\right] }$
and for $a,b\in T^{\left( \left[ p\right] \right) }(%
\mathbb{R}
^{d})$ where the sum is over all admissible cuts of the forest $\sigma $.
For more details see \cite%
{reutenauer2003free,lyons1998differential,butcher1972algebraic,connes1999hopf,gubinelli2010ramification}%
.

We equip $\mathcal{G}_{\left[ p\right] }$ with the norm $\left\vert \cdot
\right\vert :=\sum_{k=1}^{\left[ p\right] }\left\Vert \pi _{k}\left( \cdot
\right) \right\Vert ^{\frac{1}{k}}$ and define the $p$-variation of a
continuous path $g:\left[ 0,T\right] \rightarrow \mathcal{G}_{\left[ p\right]
}$ by%
\begin{equation*}
\left\Vert g\right\Vert _{p-var,\left[ 0,T\right] }:=\sup_{D,D\subset \left[
0,T\right] }\left( \tsum\nolimits_{k,t_{k}\in D}\left\vert
g_{t_{k}}^{-1}g_{t_{k+1}}\right\vert ^{p}\right) ^{\frac{1}{p}}\text{.}
\end{equation*}%
We denote by $C^{p-var}\left( \left[ 0,T\right] ,\mathcal{G}_{\left[ p\right]
}\right) $ the set of continuous paths of finite $p$-variation on $\left[ 0,T%
\right] $ taking values in $\mathcal{G}_{\left[ p\right] }$. (The exact form
of norm on $\mathcal{G}_{\left[ p\right] }$ is not important, and the
integral can be defined as long as the norm on the group and the norm on the
one-form "compensate" each other.)

For $\alpha \in L\left( T^{\left[ p\right] }\left( \mathcal{V}\right) ,%
\mathcal{U}\right) $, we denote 
\begin{equation*}
\left\Vert \alpha \left( \cdot \right) \right\Vert :=\sup_{v\in T^{\left[ p%
\right] }\left( \mathcal{V}\right) ,\left\Vert v\right\Vert =1}\left\Vert
\alpha \left( v\right) \right\Vert \text{,\ }\left\Vert \alpha \left( \cdot
\right) \right\Vert _{k}:=\sup_{v\in \mathcal{V}^{\otimes k},\left\Vert
v\right\Vert =1}\left\Vert \alpha \left( v\right) \right\Vert \text{, }%
k=1,2,\dots ,\left[ p\right] \text{.}
\end{equation*}

We say $\omega :\left\{ \left( s,t\right) |0\leq s\leq t\leq T\right\}
\rightarrow \overline{%
\mathbb{R}
^{+}}$ is a control, if $\omega $ is continuous, non-negative, vanishes on
the diagonal and satisfies $\omega \left( s,u\right) +\omega \left(
u,t\right) \leq \omega \left( s,t\right) $ for $0\leq s\leq u\leq t\leq T$.
As in\ \cite{lyons2014integration}, for $g\in C(\left[ 0,T\right] ,\mathcal{G%
}_{\left[ p\right] })$ and $\beta :\left[ 0,T\right] \rightarrow B(\mathcal{G%
}_{\left[ p\right] },\mathcal{U)}$, if the limit exists%
\begin{equation*}
\lim_{\left\vert D\right\vert \rightarrow 0,D=\left\{ t_{k}\right\}
_{k=0}^{n}\subset \left[ 0,T\right] }\beta _{0}\left(
g_{0},g_{0,t_{1}}\right) \beta _{t_{1}}\left(
g_{t_{1}},g_{t_{1},t_{2}}\right) \cdots \beta _{t_{n-1}}\left(
g_{t_{n-1}},g_{t_{n-1},T}\right) \text{ with }g_{s,t}:=g_{s}^{-1}g_{t}\text{,%
}
\end{equation*}%
then we define the limit to be the integral $\int_{0}^{T}\beta _{u}\left(
g_{u}\right) dg_{u}$.

\begin{definition}[Dominated Path]
For $g\in C^{p-var}\left( \left[ 0,T\right] ,\mathcal{G}_{\left[ p\right]
}\right) $ and Banach space $\mathcal{U}$, we say a continuous path $\rho :%
\left[ 0,T\right] \rightarrow \mathcal{U}$ is dominated by $g$, if there
exists $\beta :\left[ 0,T\right] \rightarrow B\left( \mathcal{G}_{\left[ p%
\right] },\mathcal{U}\right) $ which satisfies, for some $M>0$, control $%
\omega $ and $\theta >1$,%
\begin{gather*}
\left\Vert \beta _{t}\left( g_{t},\cdot \right) \right\Vert \leq M\text{, }%
\forall t\in \left[ 0,T\right] \text{,} \\
\left\Vert \left( \beta _{t}-\beta _{s}\right) \left( g_{t},\cdot \right)
\right\Vert _{k}\leq \omega \left( s,t\right) ^{\theta -\frac{k}{p}}\text{, }%
\forall 0\leq s\leq t\leq T\text{, }k=1,2,\dots ,\left[ p\right] \text{,}
\end{gather*}%
such that $\rho _{t}=\rho _{0}+\int_{0}^{t}\beta _{u}\left( g_{u}\right)
dg_{u}$ for $t\in \left[ 0,T\right] $.
\end{definition}

Based on the definition of dominated paths, we introduce an operator norm on
the space of one-forms to quantify the convergence of one-forms (associated
with Picard iterations).

For $g\in C^{p-var}\left( \left[ 0,T\right] ,\mathcal{G}_{\left[ p\right]
}\right) $ and control $\omega $, we say $g$ is controlled by $\omega $ if $%
\left\Vert g\right\Vert _{p-var,\left[ s,t\right] }^{p}\leq \omega \left(
s,t\right) $ for all $s<t$.

\begin{definition}[Operator Norm]
For $g\in C^{p-var}\left( \left[ 0,T\right] ,\mathcal{G}_{\left[ p\right]
}\right) $ controlled by $\omega $ and $\beta :\left[ 0,T\right] \rightarrow
B\left( \mathcal{G}_{\left[ p\right] },\mathcal{U}\right) $, we define, for $%
\gamma >1$,%
\begin{equation*}
\left\Vert \beta \right\Vert _{\gamma }:=\sup_{t\in \left[ 0,T\right]
}\left\Vert \beta _{t}\left( g_{t},\cdot \right) \right\Vert
+\max_{k=1,\dots ,\lfloor \gamma \rfloor }\sup_{0\leq s\leq t\leq T}\frac{%
\left\Vert \left( \beta _{t}-\beta _{s}\right) \left( g_{t},\cdot \right)
\right\Vert _{k}}{\omega \left( s,t\right) ^{\frac{\gamma -k}{p}}}\text{.}
\end{equation*}
\end{definition}

Suppose $\left\Vert \beta \right\Vert _{\gamma }<\infty $. When $\gamma $
increases, the integrability of $\beta $ increases. In the extreme case that 
$\gamma $ tends to infinity, $\beta $ is compelled to be a constant cocyclic
one-form, so is integrable against any continuous path. If $\gamma >p-1$ and
if there exists $\sigma :\left[ 0,T\right] \rightarrow \mathcal{U}$ such
that $\left\Vert \sigma _{t}-\sigma _{s}-\beta _{s}\left(
g_{s},g_{s,t}\right) \right\Vert \leq C\left\Vert g\right\Vert _{p-var,\left[
s,t\right] }^{\gamma }$ for all $s<t$, then $\sigma $ is a weakly controlled
path introduced by Gubinelli \cite{gubinelli2004controlling}. When $\gamma >p
$, $\beta $ is integrable against $g$ and $t\mapsto \int_{0}^{t}\beta \left(
g_{u}\right) dg_{u}$ is a dominated path.

\begin{definition}
Suppose there exists a mapping $\mathcal{I}^{\prime }\in L(T^{\left( \left[ p%
\right] \right) }\left( \mathcal{V}\right) ,T^{\left( \left[ p\right]
\right) }\left( \mathcal{V}\right) ^{\otimes 2})$ which satisfies%
\begin{equation*}
\mathcal{I}^{\prime }\left( 1\right) =\mathcal{I}^{\prime }\left( \mathcal{V}%
\right) =0\text{, }\mathcal{I}^{\prime }\left( \mathcal{V}^{\otimes
k}\right) \subseteq \mathcal{V}^{\otimes \left( k-1\right) }\otimes \mathcal{%
V}\text{, }k=2,\dots ,\left[ p\right] \text{,}
\end{equation*}%
and (with $1_{n,2}^{\prime }$ denoting the projection of $T^{\left( \left[ p%
\right] \right) }\left( \mathcal{V}\right) ^{\otimes 2}$ to $\sum_{k=1}^{%
\left[ p\right] -1}\mathcal{V}^{\otimes k}\otimes \mathcal{V}$)%
\begin{equation*}
\mathcal{I}^{\prime }\left( ab\right) =\mathcal{I}^{\prime }\left( a\right)
+1_{n,2}^{\prime }\left( \left( a\otimes a\right) \mathcal{I}^{\prime
}\left( b\right) \right) +1_{n,2}^{\prime }\left( \left( a-1\right) \otimes
\left( a\left( b-1\right) \right) \right) \text{, }\forall a,b\in \mathcal{G}%
_{\left[ p\right] }\text{.}
\end{equation*}
\end{definition}

Due to the special form of the dominated paths in Picard iterations, we only
need the mapping $\mathcal{I}^{\prime }$ (instead of $\mathcal{I}$ as in 
\cite{lyons2014integration}) for the recursive integrals to make sense.
Roughly speaking, the mapping $\mathcal{I}$ is used to define the iterated
integral of two dominated (controlled) paths, and corresponds to a universal
continuous linear mapping which has the "formal" expression:%
\begin{equation*}
\mathcal{I}\left( a\right) =\int_{0}^{T}\left( g_{0,u}-1\right) \otimes
\delta g_{0,u}\text{, }g\in C\left( \left[ 0,T\right] ,\mathcal{G}_{\left[ p%
\right] }\right) \text{, }a=g_{0,T}\text{, }\forall a\in \mathcal{G}_{\left[
p\right] }\text{.}
\end{equation*}%
The mapping $\mathcal{I}^{\prime }$ encodes part of the information of $%
\mathcal{I}$, is used to define the integral of a dominated (controlled)
path against the first level of the given group-valued path, and corresponds
to a universal continuous linear mapping with the formal expression: 
\begin{equation*}
\mathcal{I}^{\prime }\left( a\right) =\int_{0}^{T}\left( g_{0,u}-1\right)
\otimes \delta x_{u}\text{, }x:=\pi _{1}\left( g\right) \text{, }g\in
C\left( \left[ 0,T\right] ,\mathcal{G}_{\left[ p\right] }\right) \text{, }%
a=g_{0,T}\text{, }\forall a\in \mathcal{G}_{\left[ p\right] }\text{.}
\end{equation*}

In particular, $\mathcal{I}^{\prime }$ is well-defined for degree-$\left[ p%
\right] $ nilpotent Lie group and degree-$\left[ p\right] $ Butcher group
for any $p\geq 1$ (see \cite{lyons2014integration} for more explanation).

The lemma below proves that one can integrate a weakly controlled path \cite%
{gubinelli2004controlling, gubinelli2010ramification} and get a dominated
path. We made the dependence of the coefficients explicit to suit the
special needs of our proof.

\begin{lemma}
\label{Lemma controlled path is integrable}Suppose $g\in C^{p-var}\left( %
\left[ 0,T\right] ,\mathcal{G}_{\left[ p\right] }\right) $ is controlled by $%
\omega $, $\beta :\left[ 0,T\right] \rightarrow B\left( \mathcal{G}_{\left[ p%
\right] },L\left( \mathcal{V},\mathcal{W}\right) \right) $ satisfies 
\begin{equation*}
\left\Vert \beta \right\Vert _{\gamma }<\infty \text{ for some }\gamma \in
(p-1,\left[ p\right] ]\text{,}
\end{equation*}
and there exists $\varphi :\left[ 0,T\right] \rightarrow L\left( \mathcal{V},%
\mathcal{W}\right) $ which satisfies for some $M>0$, 
\begin{equation}
\left\Vert \varphi _{t}-\varphi _{s}-\beta _{s}\left( g_{s},g_{s,t}\right)
\right\Vert \leq M\left\Vert \beta \right\Vert _{\gamma }\omega \left(
s,t\right) ^{\frac{\gamma }{p}}\text{, }\forall 0\leq s<t\leq T\text{.}
\label{Definition of controlled path}
\end{equation}%
If we define $\eta :\left[ 0,T\right] \rightarrow B\left( \mathcal{G}_{\left[
p\right] },\mathcal{W}\right) $ by%
\begin{equation*}
\eta _{t}\left( a,b\right) :=\varphi _{t}\pi _{1}\left( g_{t}^{-1}a\left(
b-1\right) \right) +\beta _{t}\left( g_{t},\cdot \right) \pi _{1}\left(
\cdot \right) \mathcal{I}^{\prime }\left( g_{t}^{-1}a\left( b-1\right)
\right) \text{, }\forall a,b\in \mathcal{G}_{\left[ p\right] }\text{, }%
\forall t\in \left[ 0,T\right] \text{,}
\end{equation*}%
then for some structural constant $c$ (depending on the mapping $\mathcal{I}%
^{\prime }$), 
\begin{equation*}
\sup_{0\leq t\leq T}\left\Vert \eta _{t}\left( g_{t},\cdot \right)
\right\Vert \leq \sup_{0\leq t\leq T}\left\Vert \varphi _{t}\right\Vert
+c\left\Vert \beta \right\Vert _{\gamma }\text{,}
\end{equation*}%
and there exists a constant $C=C(M,p,\omega \left( 0,T\right) )$ such that 
\begin{equation*}
\left\Vert \left( \eta _{t}-\eta _{s}\right) \left( g_{t},\cdot \right)
\right\Vert _{k}\leq C\left\Vert \beta \right\Vert _{\gamma }\omega \left(
s,t\right) ^{\frac{\gamma +1-k}{p}}\text{, }\forall s<t\text{, }k=1,2,\dots ,%
\left[ p\right] \text{.}
\end{equation*}%
As a consequence, $\left\Vert \eta \right\Vert _{\gamma +1}<\infty $ and $%
t\mapsto \int_{0}^{t}\eta _{u}\left( g_{u}\right) dg_{u}$ is a dominated
path.
\end{lemma}

\begin{proof}
It is clear that for some constant $c$ depending on $\mathcal{I}^{\prime }$, 
\begin{equation*}
\left\Vert \eta _{t}\left( g_{t},\cdot \right) \right\Vert \leq \left\Vert
\varphi _{t}\right\Vert +c\left\Vert \beta _{t}\left( g_{t},\cdot \right)
\right\Vert \leq \sup_{0\leq t\leq T}\left\Vert \varphi _{t}\right\Vert
+c\left\Vert \beta \right\Vert _{\gamma }\text{, }\forall t\in \left[ 0,T%
\right] \text{.}
\end{equation*}

For $s<t$ and $v\in 
\mathbb{R}
\oplus \mathcal{V}\oplus \cdots \oplus \mathcal{V}^{\otimes \left[ p\right]
} $ (calculation or based on the proof in \cite{lyons2014integration}), we
have%
\begin{eqnarray}
\left( \eta _{t}-\eta _{s}\right) \left( g_{t},v\right) &=&\left( \varphi
_{t}-\varphi _{s}-\beta _{s}\left( g_{s},g_{s,t}\right) \right) \pi
_{1}\left( v\right) +\left( \beta _{t}-\beta _{s}\right) \left( g_{t},\cdot
\right) \pi _{1}\left( \cdot \right) \mathcal{I}^{\prime }\left( v\right)
\label{lemma controlled path is integrable inner1} \\
&&+\tsum\nolimits_{\sigma \in \mathcal{P}_{\left[ p\right] },\left\vert
\sigma \right\vert =\left[ p\right] }\beta _{s}\left( g_{s},\sigma \left(
g_{s,t}\right) \right) \pi _{1}\left( v\right)  \notag \\
&&+\tsum\nolimits_{k=2}^{\left[ p\right] }\tsum\nolimits_{\sigma \in 
\mathcal{P}_{\left[ p\right] },\left\vert \sigma \right\vert \geq \left[ p%
\right] +1-k}\beta _{s}\left( g_{s},\sigma \left( g_{s,t}\right) \cdot
\right) \pi _{1}\left( \cdot \right) \mathcal{I}^{\prime }\left( \pi
_{k}\left( v\right) \right) \text{.}  \notag
\end{eqnarray}%
Since $\left\Vert \beta \right\Vert _{\gamma }<\infty $ and $\mathcal{I}%
^{\prime }\left( \mathcal{V}^{\otimes k}\right) \subseteq \mathcal{V}%
^{\otimes \left( k-1\right) }\otimes \mathcal{V}$, $k=2,\dots ,\left[ p%
\right] $, we have, for some structural constant $C$ depending on the norm
of the mapping $\mathcal{I}^{\prime }$, 
\begin{equation*}
\sup_{v\in \mathcal{V}^{\otimes k},\left\Vert v\right\Vert =1}\left\Vert
\left( \beta _{t}-\beta _{s}\right) \left( g_{t},\cdot \right) \pi
_{1}\left( \cdot \right) \mathcal{I}^{\prime }\left( v\right) \right\Vert
\leq C\left\Vert \beta \right\Vert _{\gamma }\omega \left( s,t\right) ^{%
\frac{\gamma +1-k}{p}}\text{, }k=1,2,\dots ,\left[ p\right] \text{.}
\end{equation*}%
Moreover, for $s<t$,%
\begin{gather*}
\left\Vert \beta _{s}\left( g_{s},\sigma \left( g_{s,t}\right) \right)
\right\Vert \leq \left\Vert \beta \right\Vert _{\gamma }\omega \left(
s,t\right) ^{\frac{\left[ p\right] }{p}}\text{, }\forall \sigma \in \mathcal{%
P}_{\left[ p\right] }\text{, }\left\vert \sigma \right\vert =\left[ p\right] 
\text{,} \\
\left\Vert \beta _{s}\left( g_{s},\sigma \left( g_{s,t}\right) \cdot \right)
\pi _{1}\left( \cdot \right) \mathcal{I}^{\prime }\left( \pi _{k}\left(
\cdot \right) \right) \right\Vert \leq \left\Vert \beta \right\Vert _{\gamma
}\left( 1\vee \omega \left( 0,T\right) \right) \omega \left( s,t\right) ^{%
\frac{\left[ p\right] +1-k}{p}}\text{, }\forall \sigma \in \mathcal{P}_{%
\left[ p\right] }\text{, }\left\vert \sigma \right\vert \geq \left[ p\right]
+1-k\text{.}
\end{gather*}%
Hence, since $\gamma \leq \left[ p\right] $, combined with $\left( \ref%
{lemma controlled path is integrable inner1}\right) $ and $\left( \ref%
{Definition of controlled path}\right) $, for some $C\,=C(M,p,\omega \left(
0,T\right) )$, we have 
\begin{equation*}
\left\Vert \left( \eta _{t}-\eta _{s}\right) \left( g_{t},\cdot \right)
\right\Vert _{k}\leq C\left\Vert \beta \right\Vert _{\gamma }\omega \left(
s,t\right) ^{\frac{\gamma +1-k}{p}}\text{, }\forall s<t\text{, }k=1,2,\dots ,%
\left[ p\right] \text{.}
\end{equation*}
\end{proof}

For $\gamma \geq 1$, $\lfloor \gamma \rfloor $ denotes the largest integer
which is strictly less than $\gamma $. For $\sigma _{i}\in \mathcal{P}_{%
\left[ p\right] }$, $i=1,\dots ,l$, $\left\vert \sigma _{1}\right\vert
+\cdots +\left\vert \sigma _{l}\right\vert \leq \left[ p\right] $, we denote
by $\sigma _{1}\ast \cdots \ast \sigma _{l}\ $the continuous linear mapping
from $\mathcal{V}^{\otimes (\left\vert \sigma _{1}\right\vert +\cdots
+\left\vert \sigma _{l}\right\vert )}$ to $\mathcal{V}^{\otimes \left\vert
\sigma _{1}\right\vert }\otimes \cdots \otimes \mathcal{V}^{\otimes
\left\vert \sigma _{l}\right\vert }$ satisfying $\left( \sigma _{1}\ast
\cdots \ast \sigma _{l}\right) \left( a\right) =\sigma _{1}\left( a\right)
\otimes \cdots \otimes \sigma _{l}\left( a\right) $ for all $a\in \mathcal{G}%
_{\left[ p\right] }$ (see \cite{lyons2014integration} for more details).

\begin{definition}[$\protect\beta \left( f\left( \protect\rho \right)
\right) $]
\label{Definition one-form of f(rho)}Let $\rho _{\cdot }=\rho
_{0}+\int_{0}^{\cdot }\beta \left( g\right) dg:\left[ 0,T\right] \rightarrow 
\mathcal{U}$ be a dominated path and $f:\mathcal{U}\rightarrow \mathcal{W}$
be a $\limfunc{Lip}\left( \gamma \right) $ function for some $\gamma >p-1$.
We define $\beta \left( f\left( \rho \right) \right) :\left[ 0,T\right]
\rightarrow B\left( \mathcal{G}_{\left[ p\right] },\mathcal{W}\right) \ $by,
for $a,b\in \mathcal{G}_{\left[ p\right] }$ and $s\in \left[ 0,T\right] $,%
\begin{equation*}
\beta \left( f\left( \rho \right) \right) _{s}\left( a,b\right)
=\sum_{l=1}^{\lfloor \gamma \rfloor }\frac{1}{l!}\left( D^{l}f\right) \left(
f\left( \rho _{s}\right) \right) \beta _{s}\left( g_{s},\cdot \right)
^{\otimes l}\sum_{\sigma _{i}\in \mathcal{P}_{\left[ p\right] },\left\vert
\sigma _{1}\right\vert +\cdots +\left\vert \sigma _{l}\right\vert \leq \left[
p\right] }\left( \sigma _{1}\ast \cdots \ast \sigma _{l}\right) \left(
g_{s}^{-1}a\left( b-1\right) \right) \text{.}
\end{equation*}
\end{definition}

\begin{definition}[Integral]
\label{Definition of integral}Suppose $\rho :\left[ 0,T\right] \rightarrow 
\mathcal{U}\ $is a path dominated by $g\in C^{p-var}\left( \left[ 0,T\right]
,\mathcal{G}_{\left[ p\right] }\right) $ and $f:\mathcal{U}\rightarrow
L\left( \mathcal{V},\mathcal{W}\right) $ is a $\limfunc{Lip}\left( \gamma
\right) $ function for some $\gamma >p-1$. With $\beta \left( f\left( \rho
\right) \right) $ in Definition \ref{Definition one-form of f(rho)}, if we
define $\beta :\left[ 0,T\right] \rightarrow B\left( \mathcal{G}_{\left[ p%
\right] },\mathcal{W}\right) $ by%
\begin{equation}
\beta _{s}\left( a,b\right) =f\left( \rho _{s}\right) \pi _{1}\left(
g_{s}^{-1}a\left( b-1\right) \right) +\beta \left( f\left( \rho \right)
\right) _{s}\left( g_{s},\cdot \right) \otimes \pi _{1}\left( \cdot \right) 
\mathcal{I}^{\prime }\left( g_{s}^{-1}a\left( b-1\right) \right) \text{, }%
\forall a,b\in \mathcal{G}_{\left[ p\right] }\text{, }\forall s\text{,}
\label{one-form of integral}
\end{equation}%
then $\beta $ is integrable against $g$ and we define the integral $\int
f\left( \rho \right) dx:\left[ 0,T\right] \rightarrow \mathcal{W}$ by%
\begin{equation*}
\int_{0}^{t}f\left( \rho _{u}\right) dx_{u}:=\int_{0}^{t}\beta _{u}\left(
g_{u}\right) dg_{u}\text{, }\forall t\in \left[ 0,T\right] \text{.}
\end{equation*}
\end{definition}

That $\beta $ is integrable against $g$ follows from Lemma \ref{Lemma
controlled path is integrable}. When $\mathcal{G}_{\left[ p\right] }$ is the
nilpotent Lie group, the integral coincides with the first level of the
rough integral in \cite{lyons1998differential}. When $\mathcal{G}_{\left[ p%
\right] }$ is the Butcher group the integral coincides with the integral in 
\cite{gubinelli2010ramification}.

\begin{definition}[Solution]
\label{Definition of solution}For $\gamma +1>p\geq 1$, suppose $g\in
C^{p-var}\left( \left[ 0,T\right] ,\mathcal{G}_{\left[ p\right] }\right) $
and $f:\mathcal{U}\rightarrow L\left( \mathcal{V},\mathcal{U}\right) $ is a $%
\limfunc{Lip}\left( \gamma \right) $ function. We say $y$ is a solution to
the rough differential equation (with $x:=\pi _{1}\left( g\right) $)%
\begin{equation}
dy=f\left( y\right) dx\text{, }y_{0}=\xi \in \mathcal{U}\text{,}
\label{RDE1}
\end{equation}%
if $y\ $is a path dominated by $g$, and $y_{\cdot }=\xi +\int_{0}^{\cdot
}f\left( y_{u}\right) dx_{u}$ with the integral defined in Definition \ref%
{Definition of integral}.
\end{definition}

Since dominated paths are defined through integrable one-forms, instead of
formulating the fixed-point problem in the space of paths as in Definition %
\ref{Definition of solution}, we could also formulate the fixed-point
problem in the space of integrable one-forms, and $y$ is called a solution
to $\left( \ref{RDE1}\right) $ if the one-form associated with $y$ is a
fixed point of the mapping $\beta \mapsto \hat{\beta}$ where $\hat{\beta}$
is the one-form associated with $\int f\left( y\right) dx$.

\section{Existence, Uniqueness and Continuity of the Solution}

Schwartz gave a beautiful proof in \cite{schwartz1989convergence} of the
convergence of the series of Picard iterations for SDEs. Instead of working
with contraction mapping on small intervals and pasting the local solutions
together, he used the iterative expression of the differences between the $n$%
th and $\left( n+1\right) $th Picard iterations and proved that the sequence
of differences decay factorially on the whole interval. Put in the simplest
form, his argument can be summarized as follows. Suppose $f$ is $\limfunc{Lip%
}\left( 1\right) $ and consider the SDE:%
\begin{equation*}
dX_{t}=f\left( X_{t}\right) dB_{t}\text{, }X_{0}=\xi \text{.}
\end{equation*}%
We define the series of Picard iterations by $X_{t}^{n+1}=\xi
+\int_{0}^{t}f\left( X_{u}^{n}\right) dB_{u}$ with $X_{t}^{0}\equiv \xi $.
Then by using It\^{o}'s isometry and the Lipschitz property of $f$, we have%
\begin{equation*}
E\left( \left\vert X_{t}^{n+1}-X_{t}^{n}\right\vert ^{2}\right)
=E\int_{0}^{t}\left\vert f\left( X_{u}^{n}\right) -f\left(
X_{u}^{n-1}\right) \right\vert ^{2}du\leq \left\Vert f\right\Vert _{\limfunc{%
Lip}\left( 1\right) }^{2}\int_{0}^{t}E\left( \left\vert
X_{u}^{n}-X_{u}^{n-1}\right\vert ^{2}\right) du\text{.}
\end{equation*}%
By iterating this process, we obtain a factorial decay and the global
convergence of the Picard series.

We will try to extend his argument to RDEs. However, there are several
points to pay attention to: generally, $\limfunc{Lip}\left( 1\right) $ is
insufficient for rough integral to be well-defined and it is illegitimate to
take modulus inside the rough integral; there is no $L^{2}$ space and no It%
\^{o}'s isometry for general rough paths, so the factorial decay can not be
obtained in a similar way. We will rely critically on the Division Property
of Lipschitz functions, and rely critically on the factorial decay of the
Signature of a rough path \cite{lyons1998differential}. In particular, we
prove that the one-forms associated with the differences between the $n$th
and $\left( n+1\right) $th Picard iterations decay factorially in operator
norm as $n$ tends to infinity on the whole interval. As a consequence, the
one-forms associated with the Picard iterations converge in operator norm,
which implies the convergence of the Picard iterations and the convergence
of their group-valued enhancements. By using the factorial decay of the
iterated integrals, we can prove the solution is unique. The continuity of
the solution with respect to the driving noise follows from the uniform
convergence of the Picard iterations for the rough differential equations
whose driving rough paths are uniformly bounded in $p$-variation.

Let $\mathcal{U}$ and\ $\mathcal{V}$ be two Banach spaces.

\begin{definition}[Picard Iterations]
\label{Definition of yn}For $\gamma +1>p\geq 1$, suppose $g\in
C^{p-var}\left( \left[ 0,T\right] ,\mathcal{G}_{\left[ p\right] }\right) $, $%
f:\mathcal{U}\rightarrow L\left( \mathcal{V},\mathcal{U}\right) $ is a $%
\limfunc{Lip}\left( \gamma \right) $ function and $\xi \in \mathcal{U}$. We
define the series of Picard iterations associated with the rough
differential equation $dy=f\left( y\right) dx$, $y_{0}=\xi $, by 
\begin{equation*}
y_{t}^{n}:=\xi +\int_{0}^{t}f\left( y_{u}^{n-1}\right) dx_{u}\text{, }%
\forall t\in \left[ 0,T\right] \text{, with }y_{t}^{0}\equiv \xi \text{.}
\end{equation*}
\end{definition}

\begin{definition}
\label{Definition of betan}We define $\zeta ^{n}:\left[ 0,T\right]
\rightarrow B\left( \mathcal{G}_{\left[ p\right] },\mathcal{U}\right) $, $%
n\geq 1$, by 
\begin{equation*}
\zeta _{s}^{n}\left( a,b\right) =f\left( y_{s}^{n-1}\right) \pi _{1}\left(
g_{s}^{-1}a\left( b-1\right) \right) +\beta \left( f\left( y^{n-1}\right)
\right) _{s}\left( g_{s},\cdot \right) \pi _{1}\left( \cdot \right) \mathcal{%
I}^{\prime }\left( g_{s}^{-1}a\left( b-1\right) \right) ,\forall a,b\in 
\mathcal{G}_{\left[ p\right] },\forall s,
\end{equation*}%
where $\beta \left( f\left( y^{n-1}\right) \right) $ is defined in term of $%
y_{\cdot }^{n-1}=\xi +\int_{0}^{\cdot }\zeta ^{n-1}\left( g\right) dg$ as in
Definition \ref{Definition one-form of f(rho)} with $\zeta ^{0}\equiv 0$.
\end{definition}

Then based on the definition of the integral in Definition \ref{Definition
of integral}, $y_{\cdot }^{n}=\xi +\int_{0}^{\cdot }\zeta ^{n}\left(
g\right) dg$, $n\geq 0$, and $\left\{ y^{n}\right\} _{n=0}^{\infty }$ are
paths dominated by $g$.

\begin{lemma}
\label{Lemma betan are uniformly controlled}Suppose $g\in C^{p-var}\left( %
\left[ 0,T\right] ,\mathcal{G}_{\left[ p\right] }\right) $ is controlled by $%
\omega $, and $f:\mathcal{U}\rightarrow L\left( \mathcal{V},\mathcal{U}%
\right) $ is a $\limfunc{Lip}\left( \gamma \right) $ function for some $%
\gamma \in (p-1,\left[ p\right] ]$. Then there exists a constant $%
C=C(p,\gamma ,\left\Vert f\right\Vert _{\limfunc{Lip}\left( \gamma \right)
},\omega \left( 0,T\right) )$ such that 
\begin{equation*}
\sup_{n\geq 0}\left\Vert \zeta ^{n}\right\Vert _{\gamma +1}\leq C.
\end{equation*}
\end{lemma}

\begin{proof}
We first suppose $\omega \left( 0,T\right) \leq 1$, and prove that there
exists $\lambda _{p,\gamma }>0$ which only depends on $p$ and $\gamma $ such
that when $\left\Vert f\right\Vert _{\limfunc{Lip}\left( \gamma \right)
}\leq \lambda _{p,\gamma }$ we have\ $\sup_{n\geq 0}\left\Vert \zeta
^{n}\right\Vert _{\gamma +1}\leq 2\lambda _{p,\gamma }$. We prove it by
using mathematical induction. Suppose for some constant $\lambda _{n}\in
\left( 0,1\right) $,%
\begin{equation*}
\left\Vert \zeta ^{n}\right\Vert _{\gamma +1}\leq \lambda _{n}\text{,}
\end{equation*}%
which holds when $n=0$ since $\zeta ^{0}\equiv 0$. We want to prove that
there exists a constant $C_{p,\gamma }\geq 1$ such that%
\begin{equation*}
\left\Vert \zeta ^{n+1}\right\Vert _{\gamma +1}\leq \left\Vert f\right\Vert
_{\limfunc{Lip}\left( \gamma \right) }\left( 1+C_{p,\gamma }\lambda
_{n}\right) :=\lambda \left( 1+C_{p,\gamma }\lambda _{n}\right) \text{.}
\end{equation*}%
Then when $\lambda \in (0,\left( 2C_{p,\gamma }\right) ^{-1})$, if $\lambda
_{n}\leq \lambda /\left( 1-C_{p,\gamma }\lambda \right) $ then $\lambda
\left( 1+C_{p,\gamma }\lambda _{n}\right) \leq \lambda /\left( 1-C_{p,\gamma
}\lambda \right) $. Since $\lambda _{0}=0\leq \lambda /\left( 1-C_{p,\gamma
}\lambda \right) $, we have $\lambda _{n}\leq \lambda /\left( 1-C_{p,\gamma
}\lambda \right) \leq 2\lambda $ for all $n\geq 0$. It can be checked that $%
\zeta ^{n+1}$ is linear with respect to scalar multiplication of $f$, so we
assume $\left\Vert f\right\Vert _{\limfunc{Lip}\left( \gamma \right) }=1$,
and want to prove%
\begin{equation}
\left\Vert \zeta ^{n+1}\right\Vert _{\gamma +1}\leq 1+C_{p,\gamma }\lambda
_{n}\text{ \ when }\left\Vert \zeta ^{n}\right\Vert _{\lambda +1}\leq
\lambda _{n}\text{.}  \label{lemma inner 3}
\end{equation}%
By following similar proof as in \cite{lyons2014integration} of the
stability of dominated paths under composition with regular functions and by
using $\left\Vert f\right\Vert _{\limfunc{Lip}\left( \gamma \right) }=1$, $%
\omega \left( 0,T\right) \leq 1$ and $\left\Vert \zeta ^{n}\right\Vert
_{\gamma +1}\leq \lambda _{n}\in \left( 0,1\right) $, we have that there
exists $C_{p,\gamma }>0$ such that for any $s<t$,%
\begin{eqnarray*}
\left\Vert \left( \beta \left( f\left( y^{n}\right) \right) _{t}-\beta
\left( f\left( y^{n}\right) \right) _{s}\right) \left( g_{t},\cdot \right)
\right\Vert _{k} &\leq &C_{p,\gamma }\lambda _{n}\omega \left( s,t\right) ^{%
\frac{\gamma -k}{p}}\text{, }k=1,2,\dots ,\left[ p\right] -1\text{,} \\
\left\Vert f\left( y_{t}^{n}\right) -f\left( y_{s}^{n}\right) -\beta \left(
f\left( y^{n}\right) \right) _{s}\left( g_{s},g_{s,t}\right) \right\Vert
&\leq &C_{p,\gamma }\lambda _{n}\omega \left( s,t\right) ^{\frac{\gamma }{p}}%
\text{.}
\end{eqnarray*}%
Since $y_{\cdot }^{n+1}=\xi +\int_{0}^{\cdot }f\left( y^{n}\right) dx$, by
using Lemma \ref{Lemma controlled path is integrable}, we have%
\begin{eqnarray*}
\left\Vert \left( \zeta _{t}^{n+1}-\zeta _{s}^{n+1}\right) \left(
g_{t},\cdot \right) \right\Vert _{k} &\leq &C_{p,\gamma }\lambda _{n}\omega
\left( s,t\right) ^{\frac{\gamma +1-k}{p}}\text{, }\forall s<t\text{, }%
k=1,2,\dots ,\left[ p\right] \text{,} \\
\left\Vert \zeta _{t}^{n+1}\left( g_{t},\cdot \right) \right\Vert &\leq
&1+C_{p,\gamma }\lambda _{n}\text{, }\forall t\text{,}
\end{eqnarray*}%
which implies $\left( \ref{lemma inner 3}\right) $.

For the general case, we rescale the differential equation and consider $dy=%
\hat{f}\left( y\right) d\hat{x}$, $y_{0}=\xi $, with $c:=\lambda _{p,\gamma
}^{-1}||f||_{\limfunc{Lip}\left( \gamma \right) }$, $\hat{f}:=c^{-1}f$ and $%
\hat{g}:=\sum_{k=0}^{\left[ p\right] }c^{k}\pi _{k}\left( g\right) $ with $%
\hat{x}:=\pi _{1}\left( \hat{g}\right) $. Then the solution path stays
unchanged, and we have $||\hat{f}||_{\limfunc{Lip}\left( \gamma \right)
}\leq \lambda _{p,\gamma }$. If we denote by $\left\{ \beta ^{n}\right\}
_{n} $ the one-forms (as in Definition \ref{Definition of betan}) associated
with the Picard iterations of $dy=\hat{f}\left( y\right) d\hat{x}$, $%
y_{0}=\xi $, then it can be proved inductively that,%
\begin{equation*}
\zeta _{s}^{n}\left( g_{t},v\right) =\beta _{s}^{n}\left( \hat{g}_{t},\hat{v}%
\right) \text{, }\forall v\in 
\mathbb{R}
\oplus \mathcal{V}\oplus \cdots \oplus \mathcal{V}^{\otimes \left[ p\right] }%
\text{ with }\hat{v}:=\tsum\nolimits_{k=0}^{\left[ p\right] }c^{k}\pi
_{k}\left( v\right) \text{, }\forall s<t\text{, }\forall n\geq 1\text{.}
\end{equation*}%
Hence, if we can prove $\sup_{n\geq 0}\left\Vert \beta ^{n}\right\Vert
_{\gamma +1}<\infty $ then $\sup_{n\geq 0}\left\Vert \zeta ^{n}\right\Vert
_{\gamma +1}<\infty $. Denote $\hat{\omega}\left( s,t\right) :=c^{p}\omega
\left( s,t\right) $ for $s<t$. We divide the interval $\left[ 0,T\right] $
into the union of finitely many overlapping subintervals $\cup \left[
s_{i},t_{i}\right] $ such that $\hat{\omega}\left( s_{i},t_{i}\right) \leq 1$
for all $i$. Because these subintervals overlap, we can paste their
estimates together. Indeed, by using the cocyclic property, for $s<u<t$,%
\begin{equation*}
\left( \beta _{t}^{n}-\beta _{s}^{n}\right) \left( g_{t},v\right) =\left(
\beta _{t}^{n}-\beta _{u}^{n}\right) \left( g_{t},v\right) +\left( \beta
_{u}^{n}-\beta _{s}^{n}\right) \left( g_{u},g_{u,t}v\right) \text{, }\forall
v\in \mathcal{V}\oplus \cdots \oplus \mathcal{V}^{\otimes \left[ p\right] }%
\text{,}
\end{equation*}%
which implies%
\begin{eqnarray*}
\left\Vert \left( \beta _{t}^{n}-\beta _{s}^{n}\right) \left( \hat{g}%
_{t},\cdot \right) \right\Vert _{k} &\leq &\left\Vert \left( \beta
_{t}^{n}-\beta _{u}^{n}\right) \left( \hat{g}_{t},\cdot \right) \right\Vert
_{k}+\tsum\nolimits_{j=k}^{\left[ p\right] }\left\Vert \left( \beta
_{u}^{n}-\beta _{s}^{n}\right) \left( \hat{g}_{u},\cdot \right) \right\Vert
_{j} \\
&\leq &c_{1}\hat{\omega}\left( u,t\right) ^{\frac{\gamma +1-k}{p}%
}+c_{2}\tsum\nolimits_{j=k}^{\left[ p\right] }\hat{\omega}\left( s,u\right)
^{\frac{\gamma +1-j}{p}}\leq c_{3}\hat{\omega}\left( s,t\right) ^{\frac{%
\gamma +1-k}{p}}\text{,}
\end{eqnarray*}%
where $c_{i}$ may depend on $\hat{\omega}\left( 0,T\right) $.
\end{proof}

\begin{definition}
With the Picard iterations $\left\{ y^{n}\right\} _{n}$ in Definition \ref%
{Definition of yn}, we define $z^{n}:\left[ 0,T\right] \rightarrow \mathcal{U%
}$, $n\geq 1$, by 
\begin{equation*}
z_{t}^{n}=y_{t}^{n}-y_{t}^{n-1}\text{, }t\in \left[ 0,T\right] \text{.}
\end{equation*}
\end{definition}

Since $\left\{ y^{n}\right\} _{n}$ are Picard iterations which satisfy $%
y_{\cdot }^{n+1}=\xi +\int_{0}^{\cdot }f\left( y_{u}^{n}\right) dx_{u}$ with 
$y_{\cdot }^{0}\equiv \xi $, by using the division property of $f$ (i.e. $%
f\left( x\right) -f\left( y\right) =h\left( x,y\right) \left( x-y\right) $
for all $x,y\in \mathcal{U}$ and $\left\Vert h\right\Vert _{\limfunc{Lip}%
\left( \gamma -1\right) }\leq C\left\Vert f\right\Vert _{\limfunc{Lip}\left(
\gamma \right) }$), we have the recursive expression of $\left\{
z^{n}\right\} _{n}$:%
\begin{equation*}
z_{t}^{n+1}=\tint\nolimits_{0}^{t}h\left( y_{u}^{n},y_{u}^{n-1}\right)
z_{u}^{n}dx_{u}\text{, with }z_{t}^{1}=f\left( \xi \right) \left(
x_{t}-x_{0}\right) \text{, }\forall t\in \left[ 0,T\right] \text{.}
\end{equation*}%
By iteration, we have%
\begin{eqnarray*}
z_{t}^{n+1} &=&\tidotsint\nolimits_{0<u_{1}<\cdots <u_{n}<t}h\left(
y_{u_{n}}^{n},y_{u_{n}}^{n-1}\right) \cdots h\left(
y_{u_{1}}^{1},y_{u_{1}}^{0}\right) z_{u_{1}}^{1}dx_{u_{1}}\cdots dx_{u_{n}}
\\
&=&\tidotsint\nolimits_{0<u_{0}<u_{1}<\cdots <u_{n}<t}h\left(
y_{u_{n}}^{n},y_{u_{n}}^{n-1}\right) \cdots h\left(
y_{u_{1}}^{1},y_{u_{1}}^{0}\right) f\left( \xi \right)
dx_{u_{0}}dx_{u_{1}}\cdots dx_{u_{n}}\text{, }\forall t\in \left[ 0,T\right] 
\text{.}
\end{eqnarray*}

Then when $n\geq \left[ p\right] $, the increment of $z^{n}$ on a small
interval $\left[ s,t\right] $ can be approximated by a linear combination of 
$\left[ p\right] $ time-varying cocyclic one-forms, and the "coefficients"
of the cocyclic one-forms are in the form of high-ordered iterated integrals
so decay factorially as $n$ tends to infinity. Hence, by relying on the
factorial decay of the iterated integrals, we can prove inductively that the
one-forms associated with $\left\{ z^{n}\right\} _{n}$ decays factorially in
operator norm, which in turn implies the convergence in operator norm of the
one-forms associated with the Picard iterations.

\begin{definition}
\label{Definition paths of iterated integrals}For $\gamma >p\geq 1$, suppose 
$g\in C^{p-var}\left( \left[ 0,T\right] ,\mathcal{G}_{\left[ p\right]
}\right) $ with $x:=\pi _{1}\left( g\right) $ and $f:\mathcal{U}\rightarrow
L\left( \mathcal{V},\mathcal{U}\right) $ is a $\limfunc{Lip}\left( \gamma
\right) $ function. Let $h$ be the function obtained in the division
property of $f$ as in Lemma \ref{Lemma division property}. For integers $%
n\geq l\geq 0$ and $0\leq s\leq t\leq T$, we define $\eta _{s,t}^{l,n}\in
L\left( \mathcal{U},\mathcal{U}\right) $, $l\geq 1$, and $\eta
_{s,t}^{0,n}\in \mathcal{U}$, recursively by 
\begin{gather*}
\eta _{s,t}^{l,n+1}:=\int_{s}^{t}h\left( y_{u}^{n+1},y_{u}^{n}\right) \eta
_{s,u}^{l,n}dx_{u}\text{,} \\
\text{with }\eta _{s,t}^{l,l}:=\int_{s}^{t}h\left(
y_{u}^{l},y_{u}^{l-1}\right) dx_{u}\text{, }l\geq 1\text{, and }\eta
_{s,t}^{0,0}:=f\left( \xi \right) \left( x_{t}-x_{s}\right) \text{.}
\end{gather*}
\end{definition}

The integrals are well-defined based on Lemma \ref{Lemma controlled path is
integrable} and inductive arguments. In particular, we have%
\begin{equation*}
z_{t}^{n+1}=\eta _{0,t}^{0,n}\text{, }\forall t\in \left[ 0,T\right] \text{.}
\end{equation*}%
We define $\eta _{s,t}^{l,n}$ for general $l$ and $s$ to make the induction
work.

Then we define the integrable one-form $\beta _{s,\cdot }^{l,n}$ associated
with the dominated path $\eta _{s,\cdot }^{l,n}$ and prove that $\beta
_{s,\cdot }^{l,n}$ decay factorially in operator norm as $\left( n-l\right) $
tends to infinity. 

For $\sigma _{1},\sigma _{2}\in \mathcal{P}_{\left[ p\right] }$, $\left\vert
\sigma _{1}\right\vert +\left\vert \sigma _{2}\right\vert \leq \left[ p%
\right] $, we denote by $\sigma _{1}\ast \sigma _{2}\ $the continuous linear
mapping from $\mathcal{V}^{\otimes (\left\vert \sigma _{1}\right\vert
+\left\vert \sigma _{2}\right\vert )}$ to $\mathcal{V}^{\otimes \left\vert
\sigma _{1}\right\vert }\otimes \mathcal{V}^{\otimes \left\vert \sigma
_{2}\right\vert }$ satisfying $\left( \sigma _{1}\ast \sigma _{2}\right)
\left( a\right) =\sigma _{1}\left( a\right) \otimes \sigma _{2}\left(
a\right) $ for all $a\in \mathcal{G}_{\left[ p\right] }$ (see \cite%
{lyons2014integration} for more details).

\begin{definition}
\label{Definition of one-forms of iterated integrals}With $\eta _{s,t}^{l,n}$
in Definition \ref{Definition paths of iterated integrals}, for integers $%
n\geq l\geq 0$ and $s\in \lbrack 0,T)$, we define the integrable one-form $%
\beta _{s,\cdot }^{l,n}:\left[ s,T\right] \rightarrow B\left( \mathcal{G}_{%
\left[ p\right] },L\left( \mathcal{U},\mathcal{U}\right) \right) $, $l\geq 1$%
, and $\beta _{s,\cdot }^{0,n}:\left[ s,T\right] \rightarrow B\left( 
\mathcal{G}_{\left[ p\right] },\mathcal{U}\right) $ (associated with $\eta
_{s,\cdot }^{l,n}$ and $\eta _{s,\cdot }^{0,n}$ respectively) recursively
by, for $t\in (s,T]$ and $a,b\in \mathcal{G}_{\left[ p\right] }$,%
\begin{eqnarray*}
\beta _{s,t}^{l,n+1}\left( a,b\right) &=&\beta _{t,t}^{n+1,n+1}\left(
a,b\right) \eta _{s,t}^{l,n}+h\left( y_{t}^{n+1},y_{t}^{n}\right) \beta
_{s,t}^{l,n}\left( g_{t},\cdot \right) \pi _{1}\left( \cdot \right) \mathcal{%
I}^{\prime }\left( g_{t}^{-1}a\left( b-1\right) \right) \\
&&+\beta \left( h\left( y^{n+1},y^{n}\right) \right) _{t}\left( g_{t},\cdot
\right) \beta _{s,t}^{l,n}\left( g_{t},\cdot \right) \tsum\limits_{\sigma
_{i}\in \mathcal{P}_{\left[ p\right] },\left\vert \sigma _{1}\right\vert
+\left\vert \sigma _{2}\right\vert \leq \left[ p\right] }\left( \sigma
_{1}\ast \sigma _{2}\right) \left( \cdot \right) \pi _{1}\left( \cdot
\right) \mathcal{I}^{\prime }\left( g_{t}^{-1}a\left( b-1\right) \right) 
\text{,} \\
\beta _{s,t}^{l,l}\left( a,b\right) &=&h\left( y_{t}^{l},y_{t}^{l-1}\right)
\pi _{1}\left( g_{t}^{-1}a\left( b-1\right) \right) +\beta \left( h\left(
y^{l},y^{l-1}\right) \right) _{t}\left( g_{t},\cdot \right) \pi _{1}\left(
\cdot \right) \mathcal{I}^{\prime }\left( g_{t}^{-1}a\left( b-1\right)
\right) \text{, }l\geq 1\text{,} \\
\beta _{s,t}^{0,0}\left( a,b\right) &=&f\left( \xi \right) \pi _{1}\left(
g_{t}^{-1}a\left( b-1\right) \right) \text{,}
\end{eqnarray*}%
where $\beta \left( h\left( y^{n+1},y^{n}\right) \right) $ is defined from $%
\left( y^{n+1},y^{n}\right) _{t}=\left( \xi ,\xi \right) +\int_{0}^{t}\left(
\zeta _{u}^{n+1},\zeta _{u}^{n}\right) \left( g_{u}\right) dg_{u}$ as in
Definition \ref{Definition one-form of f(rho)}.
\end{definition}

The notation in the definition of $\beta ^{l,n+1}$ may need some
explanations. For $k=1,\dots ,\left[ p\right] -1$ and $v\in \mathcal{V}%
^{\otimes \left( k+1\right) }$, we have $\mathcal{I}^{\prime }\left(
v\right) \in \mathcal{V}^{\otimes k}\otimes \mathcal{V}$. Since $\sigma
_{1}\ast \sigma _{2}:\mathcal{V}^{\otimes (\left\vert \sigma _{1}\right\vert
+\left\vert \sigma _{2}\right\vert )}\rightarrow \mathcal{V}^{\otimes
\left\vert \sigma _{1}\right\vert }\otimes \mathcal{V}^{\otimes \left\vert
\sigma _{2}\right\vert }$ and $\pi _{1}:\mathcal{V}\rightarrow \mathcal{V}$,
we have $\left( \sigma _{1}\ast \sigma _{2}\right) \left( \cdot \right) \pi
_{1}\left( \cdot \right) \mathcal{I}^{\prime }\left( v\right) \in \mathcal{V}%
^{\otimes \left\vert \sigma _{1}\right\vert }\otimes \mathcal{V}^{\otimes
\left\vert \sigma _{2}\right\vert }\otimes \mathcal{V}$ for any $v\in 
\mathcal{V}^{\otimes \left( \left\vert \sigma _{1}\right\vert +\left\vert
\sigma _{2}\right\vert +1\right) }$. Then in the expression 
\begin{equation*}
\beta \left( h\left( y^{n+1},y^{n}\right) \right) _{t}\left( g_{t},\cdot
\right) \beta _{s,t}^{l,n}\left( g_{t},\cdot \right) \left( \sigma _{1}\ast
\sigma _{2}\right) \left( \cdot \right) \pi _{1}\left( \cdot \right) 
\mathcal{I}^{\prime }\left( v\right) \text{ for }v\in \mathcal{V}^{\otimes
\left( \left\vert \sigma _{1}\right\vert +\left\vert \sigma _{2}\right\vert
+1\right) }\text{,}
\end{equation*}
we treat $\beta \left( h\left( y^{n+1},y^{n}\right) \right) _{t}\left(
g_{t},\cdot \right) $ as a continuous linear mapping on $\mathcal{V}%
^{\otimes \left\vert \sigma _{1}\right\vert }$ and treat $\beta
_{s,t}^{l,n}\left( g_{t},\cdot \right) $ as a continuous linear mapping on $%
\mathcal{V}^{\otimes \left\vert \sigma _{2}\right\vert }$.

Based on the definition of integral in Definition \ref{Definition of
integral}, we have $\eta _{s,t}^{l,n}=\int_{s}^{t}\beta _{s,u}^{l,n}\left(
g_{u}\right) dg_{u}$. In particular, 
\begin{equation*}
z_{t}^{n+1}=\int_{0}^{t}\beta _{0,u}^{0,n}\left( g_{u}\right) dg_{u}\text{, }%
\forall t\in \left[ 0,T\right] \text{.}
\end{equation*}

\begin{lemma}
\label{Lemma betan is a Cauchy sequence}Suppose $g\in C^{p-var}\left( \left[
0,T\right] ,\mathcal{G}_{\left[ p\right] }\right) $ is controlled by $\omega 
$, and $f:\mathcal{U}\rightarrow L\left( \mathcal{V},\mathcal{U}\right) $ is
a $\limfunc{Lip}\left( \gamma \right) $ function for some $\gamma \in (p,%
\left[ p\right] +1]$. Then there exist a constant $C=C(p,\gamma ,\left\Vert
f\right\Vert _{\limfunc{Lip}\left( \gamma \right) },\omega \left( 0,T\right)
)\ $such that%
\begin{equation}
\left\Vert \beta _{s,\cdot }^{l,n}\right\Vert _{\gamma }\leq \frac{C^{n-%
\left[ p\right] -l}}{\left( \frac{n-\left[ p\right] -l}{p}\right) !}\text{, }%
\forall n\geq l+\left[ p\right] +1\text{, }\forall l\geq 0\text{, }\forall
s\in \lbrack 0,T)\text{,}  \label{expression factorial decay of one-forms}
\end{equation}%
where $\beta _{s,\cdot }^{l,n}$ denotes $t\mapsto \beta _{s,t}^{l,n}$
introduced in Definition \ref{Definition of one-forms of iterated integrals}
for $t\in \left[ s,T\right] $.
\end{lemma}

\begin{proof}
The constants in this proof may depend on $p$, $\gamma $, $\left\Vert
f\right\Vert _{\limfunc{Lip}\left( \gamma \right) }$ and $\omega \left(
0,T\right) $.

Firstly, we prove that, for integers $n\geq l\geq 0$ and $0\leq s\leq u\leq
t\leq T$, 
\begin{equation}
\beta _{s,t}^{l,n}\left( g_{t},v\right) =\beta _{u,t}^{l,n}\left(
g_{t},v\right) +\tsum\nolimits_{j=l+1}^{n}\beta _{u,t}^{j,n}\left(
g_{t},v\right) \eta _{s,u}^{l,j-1}\text{, }\forall v\in \mathcal{V}\oplus
\cdots \oplus \mathcal{V}^{\otimes \left[ p\right] }\text{.}
\label{expression change of time}
\end{equation}%
The equality holds when $n=l$ based on the definition of $\beta _{s,t}^{l,l}$%
. Suppose it holds when $n-l\leq s$. Then by combining the definition of $%
\beta ^{l,n+1}$ in Definition \ref{Definition of one-forms of iterated
integrals} with the inductive hypothesis $\left( \ref{expression change of
time}\right) $ and by using $\eta _{s,t}^{l,n}=\sum_{j=l+1}^{n}\eta
_{u,t}^{j,n}\eta _{s,u}^{l,j-1}+\eta _{s,u}^{l,n}+\eta _{u,t}^{l,n}$, it can
be proved that $\left( \ref{expression change of time}\right) $ holds when $%
n-l=s+1$.

Without loss of generality we assume $\gamma \in (p,\left[ p\right] +1]$.
Based on Lemma \ref{Lemma betan are uniformly controlled}, $\sup_{n\geq
0}\left\Vert \zeta ^{n}\right\Vert _{\left[ p\right] +1}<\infty $. Then
since $\left\Vert h\right\Vert _{\limfunc{Lip}\left( \gamma -1\right) }\leq
C\left\Vert f\right\Vert _{\limfunc{Lip}\left( \gamma \right) }$, by using
Lemma \ref{Lemma controlled path is integrable}, it can be proved
inductively that, for some $K_{0}\geq 1$,%
\begin{equation}
\sup_{l\geq 0}\max_{n=l,\dots ,l+\left[ p\right] }\left\Vert \beta _{s,\cdot
}^{l,n}\right\Vert _{\gamma }\leq K_{0}\text{.}
\label{inner Lemma factorial decay 1}
\end{equation}%
Then combined with $\eta _{s,t}^{l,n}=\int_{s}^{t}\beta _{s,u}^{l,n}\left(
g_{u}\right) dg_{u}$, we have, for some constant $M_{0}>0$,%
\begin{equation*}
\left\Vert \eta _{s,t}^{l,n}\right\Vert \leq M_{0}\omega \left( s,t\right) ^{%
\frac{n-l+1}{p}}\text{, }\forall s<t\text{, }n-l+1=1,2,\dots ,\left[ p\right]
\text{, }\forall l\geq 0\text{.}
\end{equation*}%
Since%
\begin{equation*}
\eta _{s,t}^{l,n}=\tsum\nolimits_{j=l+1}^{n}\eta _{u,t}^{j,n}\eta
_{s,u}^{l,j-1}+\eta _{s,u}^{l,n}+\eta _{u,t}^{l,n}\text{, }\forall 0\leq
s<u<t\leq T\text{, }\forall n\geq l\geq 0\text{,}
\end{equation*}%
by following similar proof as the factorial decay of the signature of a
rough path as in Theorem 3.7 \cite{lyons2007differential}, we have that, for 
$\beta =3p$ and some constant $M\geq 1$, (we choose $\beta =3p$ to make the
induction work) 
\begin{equation}
\left\Vert \eta _{s,t}^{l,n}\right\Vert \leq \frac{M^{\frac{n-l+1}{p}}\omega
\left( s,t\right) ^{\frac{n-l+1}{p}}}{\beta \left( \frac{n-l+1}{p}\right) !}%
\text{, }\forall s<t\text{, }\forall n\geq l\geq 0\text{.}
\label{inner Lemma factorial decay 3}
\end{equation}

Then we prove by induction on $n-l$ that, for some constants $K\geq 1$ and $%
C\geq 1$ (we will chose them in the inductive step),%
\begin{equation}
\left\Vert \left( \beta _{s,t}^{l,n}-\beta _{s,u}^{l,n}\right) \left(
g_{t},\cdot \right) \right\Vert _{k}\leq K\frac{C^{\frac{n-\left[ p\right] -l%
}{p}}\omega \left( s,t\right) ^{\frac{n-\left[ p\right] -l}{p}}}{\beta
\left( \frac{n-\left[ p\right] -l}{p}\right) !}\omega \left( u,t\right) ^{%
\frac{\gamma -k}{p}}\text{, }\forall s<u<t\text{, }\forall n\geq l+\left[ p%
\right] \text{, }\forall l\geq 0\text{,}
\label{inner Lemma factorial decay inductive hypothesis}
\end{equation}%
which holds when $n-l=\left[ p\right] $ with $K=K_{0}\beta $ based on $%
\left( \ref{inner Lemma factorial decay 1}\right) $. Suppose $\left( \ref%
{inner Lemma factorial decay inductive hypothesis}\right) $ holds when $n-l=%
\left[ p\right] ,\dots ,s$ for some $s\geq \left[ p\right] $. Then when $%
n-l=s+1$ (so $n-l\geq \left[ p\right] +1$), for $s<u<t$, based on $\left( %
\ref{expression change of time}\right) $, we have, for any $v\in \mathcal{V}%
\oplus \cdots \oplus \mathcal{V}^{\otimes \left[ p\right] }$,%
\begin{eqnarray}
&&\left( \beta _{s,t}^{l,n}-\beta _{s,u}^{l,n}\right) \left( g_{t},v\right)
\label{inner Lemma factorial decay 10} \\
&=&\left( \beta _{u,t}^{l,n}-\beta _{u,u}^{l,n}\right) \left( g_{t},v\right)
+\tsum\nolimits_{j=n-\left[ p\right] }^{n}\left( \beta _{u,t}^{j,n}-\beta
_{u,u}^{j,n}\right) \left( g_{t},v\right) \eta
_{s,u}^{l,j-1}+\tsum\nolimits_{j=l+1}^{n-\left[ p\right] -1}\left( \beta
_{u,t}^{j,n}-\beta _{u,u}^{j,n}\right) \left( g_{t},v\right) \eta
_{s,u}^{l,j-1}  \notag \\
&=&:I\left( v\right) +I\!I\left( v\right) +I\!I\!I\left( v\right) \text{.} 
\notag
\end{eqnarray}%
For $I\left( v\right) $, by using $\left( \ref{expression change of time}%
\right) $, we have%
\begin{equation*}
\left( \beta _{u,t}^{l,n}-\beta _{u,u}^{l,n}\right) \left( g_{t},v\right)
=\left( \beta _{t,t}^{l,n}-\beta _{u,u}^{l,n}\right) \left( g_{t},v\right)
+\tsum\nolimits_{j=l+1}^{n}\beta _{t,t}^{j,n}\left( g_{t},v\right) \eta
_{u,t}^{l,j-1}=\tsum\nolimits_{j=n-\left[ p\right] +1}^{n}\beta
_{t,t}^{j,n}\left( g_{t},v\right) \eta _{u,t}^{l,j-1}\text{,}
\end{equation*}%
where we used $\beta _{t,t}^{j,n}\equiv 0$ for $n\geq \left[ p\right] +j$
which can be proved inductively based on the definition of $\beta
_{s,t}^{l,n}$ in Definition \ref{Definition of one-forms of iterated
integrals}.\ Hence, for $k=1,\dots ,\left[ p\right] $, by using that $%
\left\Vert \beta _{t,t}^{j,n}\left( g_{t},\cdot \right) \right\Vert _{k}=0$, 
$j\leq n-k$, and the factorial decay of $\eta _{s,t}^{l,n}$ in $\left( \ref%
{inner Lemma factorial decay 3}\right) $, we have, for some $C_{0}\geq 1$,
(since $\gamma \leq \left[ p\right] +1$)%
\begin{eqnarray}
\left\Vert I\left( \cdot \right) \right\Vert _{k} &=&\left\Vert
\tsum\nolimits_{j=n-\left[ p\right] +1}^{n}\beta _{t,t}^{j,n}\left(
g_{t},\cdot \right) \eta _{u,t}^{l,j-1}\right\Vert _{k}
\label{inner Lemma factorial decay 11} \\
&\leq &\sum_{j=n-k+1}^{n}\left\Vert \beta _{t,t}^{j,n}\left( g_{t},\cdot
\right) \right\Vert _{k}\frac{M^{\frac{j-l}{p}}\omega \left( u,t\right) ^{%
\frac{j-l}{p}}}{\beta \left( \frac{j-l}{p}\right) !}\leq K_{0}C_{0}\frac{M^{%
\frac{n-\left[ p\right] -l}{p}}\omega \left( u,t\right) ^{\frac{n-\left[ p%
\right] -l}{p}}}{\beta \left( \frac{n-\left[ p\right] -l}{p}\right) !}\omega
\left( u,t\right) ^{\frac{\gamma -k}{p}}\text{.}  \notag
\end{eqnarray}%
For $I\!I\left( v\right) $, by using $\left( \ref{inner Lemma factorial
decay 1}\right) $ and $\left( \ref{inner Lemma factorial decay 3}\right) $,
we have%
\begin{align}
\left\Vert I\!I\left( \cdot \right) \right\Vert _{k}& =\left\Vert
\tsum\nolimits_{j=n-\left[ p\right] }^{n}\left( \beta _{u,t}^{j,n}-\beta
_{u,u}^{j,n}\right) \left( g_{t},\cdot \right) \eta
_{s,u}^{l,j-1}\right\Vert _{k}  \label{inner Lemma factorial decay 12} \\
& \leq \sum_{j=n-\left[ p\right] }^{n}\left\Vert \left( \beta
_{u,t}^{j,n}-\beta _{u,u}^{j,n}\right) \left( g_{t},\cdot \right)
\right\Vert _{k}\frac{M^{\frac{j-l}{p}}\omega \left( s,u\right) ^{\frac{j-l}{%
p}}}{\beta \left( \frac{j-l}{p}\right) !}\leq K_{0}C_{0}\frac{M^{\frac{n-%
\left[ p\right] -l}{p}}\omega \left( s,u\right) ^{\frac{n-\left[ p\right] -l%
}{p}}}{\beta \left( \frac{n-\left[ p\right] -l}{p}\right) !}\omega \left(
u,t\right) ^{\frac{\gamma -k}{p}}\text{.}  \notag
\end{align}%
For $I\!I\!I\left( v\right) $, since $\left[ p\right] <n-j\leq n-l-1=s$ when 
$j=l+1,\dots ,n-\left[ p\right] -1$, by using the inductive hypothesis $%
\left( \ref{inner Lemma factorial decay inductive hypothesis}\right) $ and
neo-classical inequality \cite{lyons2007differential, hara2010fractional},
we have 
\begin{eqnarray}
\left\Vert I\!I\!I\left( \cdot \right) \right\Vert _{k} &=&\left\Vert
\tsum\nolimits_{j=l+1}^{n-\left[ p\right] -1}\left( \beta _{u,t}^{j,n}-\beta
_{u,u}^{j,n}\right) \left( g_{t},\cdot \right) \eta
_{s,u}^{l,j-1}\right\Vert _{k}  \label{inner Lemma factorial decay 13} \\
&\leq &\sum_{j=l+1}^{n-\left[ p\right] -1}K\frac{C^{\frac{n-\left[ p\right]
-j}{p}}\omega \left( u,t\right) ^{\frac{n-\left[ p\right] -j}{p}}}{\beta
\left( \frac{n-\left[ p\right] -j}{p}\right) !}\frac{M^{\frac{j-l}{p}}\omega
\left( s,u\right) ^{\frac{j-l}{p}}}{\beta \left( \frac{j-l}{p}\right) !}%
\omega \left( u,t\right) ^{\frac{\gamma -k}{p}}  \notag \\
&\leq &K\frac{p}{\beta }\frac{\left( C\vee M\right) ^{\frac{n-\left[ p\right]
-l}{p}}\omega \left( s,t\right) ^{\frac{n-\left[ p\right] -l}{p}}}{\beta
\left( \frac{n-\left[ p\right] -l}{p}\right) !}\omega \left( u,t\right) ^{%
\frac{\gamma -k}{p}}\text{.}  \notag
\end{eqnarray}%
Hence, based on $\left( \ref{inner Lemma factorial decay 10}\right) $, $%
\left( \ref{inner Lemma factorial decay 11}\right) $, $\left( \ref{inner
Lemma factorial decay 12}\right) $ and $\left( \ref{inner Lemma factorial
decay 13}\right) $, since $n-\left[ p\right] -l\geq 1$ and $\beta =3p$, by
choosing $K=K_{0}\left( C_{0}\vee \beta \right) $ ($\beta $ to take into
account of $n=l+\left[ p\right] $) and $C=3M$, we have $\left( \ref{inner
Lemma factorial decay inductive hypothesis}\right) $ holds when $n-l=s+1$,
and the induction is complete.

On the other hand, when $n-l\geq \left[ p\right] $, based on $\left( \ref%
{expression change of time}\right) $ and by using $\beta _{t,t}^{j,n}\equiv
0 $ for $n\geq \left[ p\right] +j$, we have%
\begin{equation*}
\beta _{s,t}^{l,n}\left( g_{t},\cdot \right) =\beta _{t,t}^{l,n}\left(
g_{t},\cdot \right) +\tsum\nolimits_{j=l+1}^{n}\beta _{t,t}^{j,n}\left(
g_{t},\cdot \right) \eta _{s,t}^{l,j-1}=\tsum\nolimits_{j=n-\left[ p\right]
+1}^{n}\beta _{t,t}^{j,n}\left( g_{t},\cdot \right) \eta _{s,t}^{l,j-1}\text{%
.}
\end{equation*}%
Hence, by using the factorial decay of $\eta _{s,t}^{l,n}$ in $\left( \ref%
{inner Lemma factorial decay 3}\right) $, we have%
\begin{equation}
\left\Vert \beta _{s,t}^{l,n}\left( g_{t},\cdot \right) \right\Vert \leq
\sum_{j=n-\left[ p\right] +1}^{n}\left\Vert \beta _{t,t}^{j,n}\left(
g_{t},\cdot \right) \right\Vert _{k}\left\Vert \eta
_{s,t}^{l,j-1}\right\Vert \leq K_{0}C_{0}\frac{M^{\frac{n-\left[ p\right] -l%
}{p}}\omega \left( s,t\right) ^{\frac{n-\left[ p\right] -l}{p}}}{\left( 
\frac{n-\left[ p\right] -l}{p}\right) !}\text{.}
\label{inner Lemma factorial decay uniform bound}
\end{equation}

Then for $s\in \lbrack 0,T)$ since 
\begin{equation*}
\left\Vert \beta _{s,\cdot }^{l,n}\right\Vert _{\gamma }:=\sup_{s\leq t\leq
T}\left\Vert \beta _{s,t}^{l,n}\left( g_{t},\cdot \right) \right\Vert
+\max_{k=1,\dots ,\left[ p\right] }\sup_{s\leq u\leq t\leq T}\omega \left(
u,t\right) ^{-\left( \gamma -\frac{k}{p}\right) }\left\Vert \left( \beta
_{s,t}^{l,n}-\beta _{s,u}^{l,n}\right) \left( g_{t},\cdot \right)
\right\Vert _{k}\text{,}
\end{equation*}%
we have the lemma holds based on $\left( \ref{inner Lemma factorial decay
inductive hypothesis}\right) $ and $\left( \ref{inner Lemma factorial decay
uniform bound}\right) $.
\end{proof}

\begin{theorem}[Existence, Uniqueness and Continuity of the Solution]
\label{Theorem existence uniqueness stability}For $\left[ p\right] +1\geq
\gamma >p\geq 1$, suppose $g\in C^{p-var}\left( \left[ 0,T\right] ,\mathcal{G%
}_{\left[ p\right] }\right) $ is controlled by $\omega $, $f:\mathcal{U}%
\rightarrow L\left( \mathcal{V},\mathcal{U}\right) $ is a $\limfunc{Lip}%
\left( \gamma \right) $ function and $\xi \in \mathcal{U}$. Then the Picard
iterations $\left\{ y^{n}\right\} _{n=0}^{\infty }$ in Definition \ref%
{Definition of yn} converge uniformly on $\left[ 0,T\right] $ to the unique
solution to the rough differential equation%
\begin{equation*}
dy=f\left( y\right) dx\text{, }y_{0}=\xi \text{,}
\end{equation*}%
and the solution is continuous with respect to $g$ in $p$-variation norm.
Moreover, there exist integrable one-forms $\beta ^{n}:\left[ 0,T\right]
\rightarrow B\left( \mathcal{G}_{\left[ p\right] },\mathcal{U}\right) $, $%
n\geq 0$, and a constant $C=C(p,\gamma ,\left\Vert f\right\Vert _{\limfunc{%
Lip}\left( \gamma \right) },\omega \left( 0,T\right) )>0$\ such that 
\begin{gather}
y_{t}^{n}=\xi +\int_{0}^{t}\beta _{u}^{n}\left( g_{u}\right) dg_{u}\text{, }%
\forall t\in \left[ 0,T\right] \text{, }  \notag \\
\text{and }\left\Vert \beta ^{n+1}-\beta ^{n}\right\Vert _{\gamma }\leq 
\frac{C^{n-\left[ p\right] }}{\left( \frac{n-\left[ p\right] }{p}\right) !}%
\text{, }\forall n\geq \left[ p\right] +1\text{.}  \label{factorial decay}
\end{gather}
\end{theorem}

There are some remarks.

\begin{enumerate}
\item In proving the convergence of the Picard iterations, we proved the
convergence in operator norm of their associated one-forms. In particular,
we proved that the one-form associated with the difference between the $n$th
and $\left( n+1\right) $th Picard iterations decays factorially on $\left[
0,T\right] $ as $n$ tends to infinity.

\item Let $\rho :\left[ 0,T\right] \rightarrow \mathcal{W}$ be a path
dominated by $g$. Then the integral of $\rho $ against $y^{n}$ is well
defined:%
\begin{equation*}
\tint\nolimits_{0}^{t}\rho _{u}\otimes dy_{u}^{n}=\tint\nolimits_{0}^{t}\rho
_{u}\otimes f\left( y_{u}^{n-1}\right) dx_{u},\forall t\in \left[ 0,T\right]
.
\end{equation*}%
In particular, since $y^{n}$ is a dominated path, there exists a canonical
enhancement of $y^{n}$ to a group-valued path, which could take values in
nilpotent Lie group or Butcher group.

\item When treated as a Banach space-valued path, the group-valued
enhancement is again a dominated path. Since the one-form associated with
the enhancement is continuous with respect to the one-form associated with
the base dominated path, the one-forms of the enhancement of $y^{n}$ also
converge in operator norm, which implies the uniform convergence of the
group-valued enhancements.

\item When $f$ is $\limfunc{Lip}\left( \gamma \right) $ for $\gamma >p-1$,
the one-forms associated with the Picard iterations are uniformly bounded.
When the dimension is finite, based on Arzel\`{a}-Ascoli theorem, there
exists a subsequence of the one-forms which converges, so the associated
paths (and their enhancements) converge to a solution.

\item When $f$ is locally Lipschitz and the dimension is finite, the
solution exists (uniquely) up to explosion. Indeed, by Whitney's extension
theorem, the restriction of $f$ to any compact set can be extended to a
global Lipschitz function without increasing its Lipschitz norm, so the
solution exists up to exit time of that compact set. For similar reason,
when $f$ is locally $\limfunc{Lip}\left( \gamma \right) $ for $\gamma >p$,
any two solutions must agree on any compact set, so the solution exists
uniquely up to explosion.
\end{enumerate}

\begin{proof}
Suppose $\left\{ y^{n}\right\} _{n}$ are the Picard iterations in Definition %
\ref{Definition of yn}. Since $z^{i+1}=y^{i+1}-y^{i}$ and $\beta ^{0,i}$ is
the integrable one-form associated with $z^{i+1}$, if we define $\beta ^{n}:%
\left[ 0,T\right] \rightarrow B\left( \mathcal{G}_{\left[ p\right] },%
\mathcal{U}\right) $, $n\geq 1$, by 
\begin{equation*}
\beta _{s}^{n}\left( a,b\right) =\tsum\nolimits_{i=0}^{n-1}\beta
_{s}^{0,i}\left( a,b\right) ,\forall a,b\in \mathcal{G}_{\left[ p\right]
},\forall s\in \left[ 0,T\right] \text{,}
\end{equation*}%
then $\beta ^{n}$ is integrable and 
\begin{equation*}
y_{t}^{n}=\xi +\tint\nolimits_{0}^{t}\beta _{u}^{n}\left( g_{u}\right) dg_{u}%
\text{, }\forall t\in \left[ 0,T\right] \text{, }\forall n\geq 1\text{.}
\end{equation*}%
(Since $y^{0}\equiv \xi $, we set $\beta ^{0}\equiv 0$ so $y_{\cdot
}^{0}=\xi +\int_{0}^{\cdot }\beta ^{0}\left( g\right) dg$.) Based on Lemma %
\ref{Lemma betan is a Cauchy sequence}, we have $\left( \ref{factorial decay}%
\right) $ holds and $\beta ^{n}$ converge in operator norm as $n$ tends to
infinity (denote the limit by $\beta $), so $y_{\cdot }^{n}=\xi
+\int_{0}^{\cdot }\beta ^{n}\left( g\right) dg$ converge uniformly to $%
y_{\cdot }:=\xi +\int_{0}^{\cdot }\beta \left( g\right) dg$. Moreover, by
using the division property of $f$ (i.e.\ $f\left( x\right) -f\left(
y\right) =h\left( x,y\right) \left( x-y\right) $ for all $x,y$ in $\mathcal{U%
}$ and $\left\Vert h\right\Vert _{\limfunc{Lip}\left( \gamma -1\right) }\leq
C\left\Vert f\right\Vert _{\limfunc{Lip}\left( \gamma \right) }$), we have%
\begin{eqnarray*}
y_{t}^{n+1}-y_{t}^{n} &=&z_{t}^{n+1}=\tint\nolimits_{0}^{t}h\left(
y_{u}^{n},y_{u}^{n-1}\right) z_{u}^{n}dx_{u} \\
&=&\tint\nolimits_{0}^{t}h\left( y_{u}^{n},y_{u}^{n-1}\right) \left(
y_{u}^{n}-y_{u}^{n-1}\right) dx_{u}=\tint\nolimits_{0}^{t}\left( f\left(
y_{u}^{n}\right) -f\left( y_{u}^{n-1}\right) \right) dx_{u}\text{, }\forall
t\in \left[ 0,T\right] \text{.}
\end{eqnarray*}%
Hence,%
\begin{equation*}
y_{t}^{n+1}=\xi +\tint\nolimits_{0}^{t}f\left( y_{u}^{n}\right) dx_{u}\text{%
, }\forall t\in \left[ 0,T\right] \text{, }\forall n\geq 0\text{, with }%
y^{0}\equiv \xi \text{.}
\end{equation*}%
Since both $y^{n}$ and $y^{n+1}$ are dominated paths and their associated
one-forms converge to $\beta $ as $n$ tends to infinity, by letting $%
n\rightarrow \infty $ on both sides, we have $\beta \ $is the fixed point of
the mapping $\beta \mapsto \hat{\beta}$ where $\hat{\beta}$ is the one-form
associated with the dominated path $t\mapsto \int_{0}^{t}f\left( y\right) dx$%
. Hence, $y$ is a dominated path satisfying the integral equation and $y$ is
a solution.

Then we prove that the solution is unique. Suppose \thinspace $\hat{y}$ is
another solution. By using the division property of $f$, we have%
\begin{equation*}
y_{t}-\hat{y}_{t}=\tint\nolimits_{0}^{t}\left( f\left( y_{u}\right) -f\left( 
\hat{y}_{u}\right) \right) dx_{u}=\tint\nolimits_{0}^{t}h\left( y_{u},\hat{y}%
_{u}\right) \left( y_{u}-\hat{y}_{u}\right) dx_{u}\text{,\ }\forall t\in %
\left[ 0,T\right] \text{.}
\end{equation*}%
By iterating this process, we have, for any integer $n\geq 1$,%
\begin{equation*}
y_{t}-\hat{y}_{t}=\tidotsint\nolimits_{0<u_{1}<\cdots <u_{n}<t}h\left(
y_{u_{n}},\hat{y}_{u_{n}}\right) \cdots h\left( y_{u_{1}},\hat{y}%
_{u_{1}}\right) \left( y_{u_{1}}-\hat{y}_{u_{1}}\right) dx_{u_{1}}\cdots
dx_{u_{n}}\text{,\ }\forall t\in \left[ 0,T\right] \text{.}
\end{equation*}%
Since $\left( y,\hat{y}\right) $ is a dominated path and $h$ is a $\limfunc{%
Lip}\left( \gamma -1\right) $ function, we can define based on Lemma \ref%
{Lemma controlled path is integrable} the dominated paths $\rho ^{n}:\left[
0,T\right] \rightarrow L\left( \mathcal{U},\mathcal{U}\right) $, $n\geq 1$,
recursively by%
\begin{equation*}
\rho _{t}^{n+1}=\tint\nolimits_{0}^{t}h\left( y_{u},\hat{y}_{u}\right) \rho
_{u}^{n}dx_{u}\text{ with }\rho _{t}^{1}=\tint\nolimits_{0}^{t}h\left( y_{u},%
\hat{y}_{u}\right) dx_{u}\text{,\ }\forall t\in \left[ 0,T\right] \text{,}
\end{equation*}%
and we have%
\begin{equation*}
y_{t}-\hat{y}_{t}=\tint\nolimits_{0}^{t}\rho _{u}^{n}\left( y_{u}-\hat{y}%
_{u}\right) dx_{u}\text{, }\forall t\in \left[ 0,T\right] \text{, }\forall
n\geq 1\text{.}
\end{equation*}%
Then by following similar proof to that of Lemma \ref{Lemma betan is a
Cauchy sequence}, the one-form associated with $\rho ^{n}$ decays
factorially. Since $y-\hat{y}$ is another dominated path, the one-form
associated with the dominated path $\tint\nolimits_{0}^{\cdot }\rho
_{u}^{n}\left( y_{u}-\hat{y}_{u}\right) dx_{u}$ also decays factorially,
which implies that $y=\hat{y}$.

It is clear that for any integer $n\geq 1$, the mapping $g\mapsto \beta ^{n}$
is continuous. Suppose $g^{m}\rightarrow g$ in $p$-variation norm, then by
uniform convergence of the mapping $\beta ^{n}\mapsto \beta $ with respect
to the $p$-variation of $g$ (based on Lemma \ref{Lemma betan is a Cauchy
sequence}), we have $g\mapsto \beta $ is continuous, which implies that the
mapping $g\mapsto y$ is continuous with respect to $g$ in $p$-variation norm.
\end{proof}

\begin{flushright}
The Oxford-Man Institute, University of Oxford
\end{flushright}

\bibliographystyle{abbrv}
\bibliography{acompat,roughpath}

\end{document}